\documentclass[final]{siamltex}
\usepackage{amssymb}
\usepackage{graphicx}
\usepackage[numbers,sort&compress]{natbib}
\usepackage{color}

\title{A Spectral Analysis of Subspace Enchanced Preconditioners}

\author{Tao Zhao\thanks{Department of Computer Science, University of Colorado Boulder, CO 80309, USA (tao.zhao@colorado.edu)}}

\begin{document}
\maketitle

\begin{abstract}
It is well-known that the convergence of Krylov subspace methods to solve linear system depends on the spectrum of the coefficient matrix, moreover, it is widely accepted that for both symmetric and unsymmetric systems Krylov subspace methods will converge fast if the spectrum of the coefficient matrix is clustered. In this paper we investigate the spectrum of the system preconditioned by the deflation, coarse correction and adapted deflation preconditioners. Our analysis shows that the spectrum of the preconditioned system is highly impacted by the angle between the coarse space for the construction of the three preconditioners and the subspace spanned by the eigenvectors associated with the small eigenvalues of the coefficient matrix. Furthermore, we prove that the accuracy of the inverse of projection matrix also impacts the spectrum of the preconditioned system. Numerical experiments emphasized the theoretical analysis.
\end{abstract}


%\begin{AMS}\end{AMS}

\pagestyle{myheadings}
\thispagestyle{plain}

\section{Introduction}
We consider the iterative solution of a linear system
$$Ax=b,$$
where $A\in\mathbb{R}^{n\times n}$ is symmetric and positive definite (SPD). It is well-known that the convergence of Krylov subspace methods for solving linear systems depend on the eigenvalue distribution of $A$. Recently, several studies \cite{NatafHuaVic, NatafMikolajZhao, TangNVE, deSturler, MorganDeflation, Padiy, ErlanggaSisc, ErlanggaSimax, NabbenVuik, Giraud, Gosselet} have shown that by removing the subspace spanned by the eigenvectors corresponding to several small eigenvalues from the Krylov search space makes the spectrum more clustered and consequently the convergence is improved. %Preconditioners based on the deflation and coarse space correction techniques are two main types of the spectral preconditioners that deflate or shift small eigenvalues.

In this paper, we refer to matrix of the form
\begin{eqnarray}\label{de:projection}
E=Z^TAZ,~~Z\in\mathbb{R}^{n\times r}.
\end{eqnarray}
as a projection matrix, and the subspace spanned by the columns of $Z$ is referred to as a coarse space. In an ideal situation, the coarse space contains the vectors corresponding to the lower part of the spectrum that is responsible for the stagnation of Krylov subspace methods. As shown in this paper, the preconditioned systems will have the desired properties when the preconditioner is enchanced with the ideal coarse subspace. In contrast to the general coarse space, we refer to the coarse space spanned by the eigenvectors associated with several small eigenvalues of $A$ as the exact coarse space.

Next we briefly mention a few existing approaches that take the form of a standard precondition with an additive coarse space enchancement. In \cite{ErlanggaSimax,NabbenVuik}, the deflation preconditioner is defined by
\begin{eqnarray}\label{de:P_D}
P_D=I-AZE^{-1}Z^{T}.
\end{eqnarray}
Obviously $P_DA$ is singular since $P_D$ is singular. Fortunately, Krylov subspace methods converge for singular linear systems as long as they are consistent; furthermore, zero eigenvalues do not impact the convergence since the corresponding eigenvectors never enter the Krylov subspace \cite{NabbenVuik}. %To measure the convergence rate of Krylov subspace methods for symmetric positive semidefinite system, the effective condition number\cite{NabbenVuik}, $\kappa_{eff}$, is defined by the ratio of the maximal to the minimal nonzero eigenvalues.

Instead of zero out the small eigenspace in the deflation method, the preconditioners based on coarse correction shift the small eigenvalues to the large ones. In \cite{TangNVE}, the coarse correction preconditioner in domain decomposition method is defined by
\begin{eqnarray}\label{de:P_C}
P_C=I+ZE^{-1}Z^{T}.
\end{eqnarray}
The abstract additive coarse correction is $M^{-1}+ZE^{-1}Z^T$, where $M$ is the sum of the local solves in each subdomain. Another popular preconditioner in domain decomposition method is the abstract balancing preconditioner \cite{ErlanggaSimax}
\begin{eqnarray}\label{de:P_BNN}
P_{BNN}=(I-ZE^{-1}Z^{T}A)M^{-1}(I-AZE^{-1}Z^{T})+ZE^{-1}Z^{T}.
\end{eqnarray}
The adapted deflation preconditioner \cite{TangNVE} is defined as
\begin{eqnarray*}
P_{ADEF1}=M^{-1}P_D+ZE^{-1}Z^{T}.
\end{eqnarray*}
It is shown in \cite{TangNVE} that $P_{ADEF1}$ is cheaper than $P_{BNN}$ but is as robust as $P_{BNN}$. Moreover, $P_{BNN}A$ and $P_{ADEF1}A$ have an identical spectrum. In both $P_{BNN}$ and $P_{ADEF1}$, the first term corresponds to the fine space and the second term is for the coarse space, therefore they are called two-level preconditioners. Throughout this paper, we restrict our analysis to one-level methods. Let $M$ in $P_{ADEF1}$ be $I$. We define $P_A$ as
\begin{eqnarray}\label{de:P_A-DEF1}
P_A=I-AZE^{-1}Z^T+ZE^{-1}Z^{T}.
\end{eqnarray}

Let $X$ be any basis of a coarse space not $Z$. Obviously the following identity
$$X(X^TAX)^{-1}X^T=Z(Z^TAZ)^{-1}Z^T$$
holds in exact arithmetic since there exists a nonsingular matrix $C\in\mathbb{R}^{r\times r}$ such that $X=ZC$. This implies that $P_D$, $P_C$ and $P_A$ are determined uniquely by the coarse space. Thus we can choose an appropriate basis to form $Z$ and then to construct a preconditioner with certain desirable properties.

As we will see, if an approximate coarse space is used to construct preconditioners, then the spectrum of the preconditioned systems is related to the angle between the approximate and exact coarse spaces. We also prove that the coarse correction and adapted deflation preconditioners are more roust than the deflation preconditioner when the projection matrix is solved inexactly. In section 2, we first review the spectral properties of the preconditioned system when the preconditioners are constructed with the exact coarse space, then we estimate the spectral bounds of the preconditioned systems when the approximate coarse space is used. Section 3 presents the perturbation analysis on the spectrum of the preconditioned system in the case that the projection matrix have some perturbation. Numerical results are reported in Section 4.

\section{Coarse space spanned by the approximate coarse space}\label{se:spectrum}
In this section, we briefly review the spectrum of the system preconditioned by using the exact coarse space, then based on these properties we derive the bounds of the spectrum of the system preconditioned by using the approximate coarse space. Let $(\lambda_i, v_i)$ be an eigenpair of $A$ and $v_i$ be normalized. $(v_i,\cdots,v_n)$ is orthogonal since $A$ is SPD. The spectral decomposition of $A$ can be written as
\begin{eqnarray}
A=(V,V_\perp)\left(\begin{array}{cc}\Lambda & 0\\ 0 & \Lambda_\perp\end{array}\right)\left(\begin{array}{c}V^T\\V^T_\perp\end{array}\right),
\end{eqnarray}
where $\Lambda=diag\{\lambda_1,\cdots,\lambda_r\}$, $\Lambda_\perp=diag\{\lambda_{r+1},\cdots,\lambda_n\}$, $V=(v_1,\cdots,v_r)$, and $V_\perp=(v_{r+1},\cdots,v_n)$.

Assume that $\lambda_1$, $\lambda_2$, $\cdots$, $\lambda_r$ are small eigenvalues that impacts the convergence of the Krylov subspace methods. Theorem \ref{th:exact_coarse_space} and Theorem \ref{th:exactright} show that if we take $Z=V$, the small eigenvalues are removed or shifted when the preconditioners are applied on either side of $A$, moreover, the rest of eigenvalues and all the eigenvectors are not changed.
\begin{theorem}\label{th:exact_coarse_space}
Let $\tilde E=V^TAV$. Define 
\begin{eqnarray*}
  \tilde P_D &=& I-AV\tilde E^{-1}V^T,\\
  \tilde P_C &=& I+V\tilde E^{-1}V^T,\\
  \tilde P_A &=& I-AV\tilde E^{-1}V^T+V\tilde E^{-1}V^T.
\end{eqnarray*}
Then we have the following spectral decomposition
\begin{eqnarray*}
  \tilde P_DA &=& (V,V_\perp)\left(\begin{array}{cc} 0 & 0\\ 0 & \Lambda_\perp\end{array}\right)\left(\begin{array}{c}V^T\\V^T_\perp\end{array}\right), \\
  \tilde P_CA &=& (V,V_\perp)\left(\begin{array}{cc} I+\Lambda & 0\\ 0 & \Lambda_\perp\end{array}\right)\left(\begin{array}{c}V^T\\V^T_\perp\end{array}\right), \\
  \tilde P_AA &=& (V,V_\perp)\left(\begin{array}{cc} I & 0\\ 0 & \Lambda_\perp\end{array}\right)\left(\begin{array}{c}V^T\\V^T_\perp\end{array}\right).
\end{eqnarray*}
\end{theorem}
\begin{proof}
 From the definition of $\tilde E$, we have $$\tilde E=V^TAV=V^TV\Lambda=\Lambda.$$
 Next, let us consider the first result. Obviously, $\tilde P_DAV=0$. 
 We then have
 \begin{eqnarray*}
    \tilde P_DAV_\perp &=& AV_\perp-AV\Lambda^{-1}V^TAV_\perp\\
                       &=& V_\perp\Lambda_\perp-VV^TV_\perp\Lambda_\perp\\
                       &=& V_\perp\Lambda_\perp.
 \end{eqnarray*}
 Thus the first spectral decomposition holds. Since 
 \begin{eqnarray*}
   \tilde P_CA = A+AV\tilde E^{-1}V^T = A+V\Lambda\Lambda^{-1} V^T = A+VV^T,
 \end{eqnarray*}
 we have $$\tilde P_CAV=AV+VV^TV=V(\Lambda+I)$$ and $$\tilde P_CAV_\perp=AV_\perp+VV^TV_\perp=AV_\perp=V_\perp\lambda_\perp.$$
 Hence, the second spectral decomposition is true. The third spectral decomposition follows from 
 $$\tilde P_AAV=AV-AV\tilde E^{-1}V^TAV+V\tilde E^{-1}V^TAV=V$$ and
 \begin{eqnarray*}
   \tilde P_AAV_\perp &=& AV_\perp-AV\tilde E^{-1}V^TAV_\perp+V\tilde E^{-1}V^TAV_\perp\\
                      &=& AV_\perp-AV\tilde E^{-1}V^TV_\perp\Lambda_\perp+V\tilde E^{-1}V^TV_\perp\Lambda_\perp\\
                      &=& AV_\perp = V_\perp\Lambda_\perp
 \end{eqnarray*}
 It follows immediately from these three spectral decomposition that $P_DA$, $P_CA$ and $P_AA$ are symmetric.
\end{proof}

%Let $\kappa$ be the spectral condition number defined in \cite{Saad}. Let the effective condition number, $\kappa_{eff}$, be the ratio of the maximal to the minimal nonzero eigenvalues \cite{NabbenVuik}. From Theorem \ref{th:exact_coarse_space}, we have $\kappa_{eff}(\tilde P_DA)\le\kappa(\tilde P_CA)$ and $\kappa_{eff}(\tilde P_DA)\le\kappa(\tilde P_AA)$. Theorem \ref{th:exactright} shows that if the three preconditioners are applied on the right side of $A$, the preconditioned systems have the same spectrum decomposition as described in Theorem \ref{th:exact_coarse_space}. 
\begin{theorem}\label{th:exactright}
 $\tilde P_D$, $\tilde P_C$ and $\tilde P_A$ are defined in Theorem \ref{th:exact_coarse_space}. Then we have $A\tilde P_D=\tilde P_DA$, $A\tilde P_C=\tilde P_CA$ and $A\tilde P_A=\tilde P_AA$.
\end{theorem}
\begin{proof}
 The first result follows from
 $$A\tilde P_DV=AV-AV\Lambda\Lambda^{-1}V^TV=0=\tilde P_DAV$$ and
 $$A\tilde P_DV_\perp=AV_\perp+AV\Lambda^{-1}V^TV_\perp=V_\perp\Lambda_\perp=\tilde P_DAV_\perp.$$
 Note that $A$, $\tilde P_C$ and $\tilde P_CA$ are symmetric. We then have
 $$\tilde P_CA=(\tilde P_CA)^T=A^T\tilde P^T_C=A\tilde P_C.$$
 Since
 \begin{eqnarray*}
   A\tilde P_A &=& A-A^2V\tilde E^{-1}V^T+AV\tilde E^{-1}V^T\\
               &=& A-AV\Lambda\Lambda^{-1}V^T+V\Lambda\Lambda^{-1}V^T\\
               &=& A-AVV^T+VV^T,
 \end{eqnarray*}
 we get $$A\tilde P_AV=AV-AVV^TV+VV^TV=V=\tilde P_AAV$$ and $$A\tilde P_AV_\perp=AV_\perp+AVV^TV_\perp+VV^TV_\perp=AV_\perp=V_\perp\Lambda_\perp=\tilde P_AAV_\perp.$$
 The third result is therefore true. In addition, $AP_D$, $AP_C$ and $AP_A$ are symmetric as well.
\end{proof}

For a large system, it is impractical to build the preconditioners by using the exact eigenvectors associated with the small eigenvalues, since in general computing these eigenvectors is more costly and more difficult than solving a linear system. Nevertheless, for some cases, the approximate eigenvectors are cheaply obtained and thus are used to produce the preconditioners. For example, in the Newton method for solving nonlinear problems, the information obtained during solving the first linear system is reused to build the coarse spaces for accelerating the convergence of the succeeding linear systems, see \cite{Gosselet} and \cite{NatafMikolajZhao}. We hope that the system preconditioned by using the approximate coarse space has the similar eigenvalue distribution as described in Theorem \ref{th:exact_coarse_space}. This motivates us to analyse how the perturbation in the coarse space impacts the spectrum of the system preconditioned by $P_D$, $P_C$ and $P_A$. 

In the following discussion we assume that $Z$ is column orthogonal. Let $\mathcal{Z}$ be the subspace spanned by the columns of $Z$. Let $\mathcal{Z}^\perp$ be the orthogonal complement of $\mathcal{Z}$ and $Z_\perp$ be an orthogonal basis of $\mathcal{Z}^\perp$. Likewise, let $\mathcal{V}$ be the subspace spanned by the columns of $V$, $\mathcal{V}^\perp$ be the orthogonal complement of $\mathcal{V}$ and $V_\perp$ be an orthogonal basis of $\mathcal{V}^\perp$. Let $\sigma$ denote the singular value of a matrix. Let $dist(\mathcal{Z},\mathcal{V})$ denote the distance between subspaces $\mathcal{Z}$ and $\mathcal{V}$. It is shown in \cite{Gene} that $dist(\mathcal{Z},\mathcal{V})$ can be evaluated by either $\sigma_{max}(Z^TV_\perp)$ or $\sigma_{max}(V^TZ_\perp)$. Note that $0\le dist(\mathcal{Z},\mathcal{V})\le1$ since $\mathcal{Z}$ and $\mathcal{V}$ have the same dimension. We define the acute angle between subspaces $\mathcal{Z}$ and $\mathcal{V}$ as $\theta=\arcsin dist(\mathcal{Z}, \mathcal{V})$.
The next lemma shows that $\cos\theta$ can be evaluated in a similar way as $\sin\theta$.
\begin{lemma}\label{le:angle}
Let $\theta$ be the acute angle between subspaces $\mathcal{Z}$ and $\mathcal{V}$ that have the same dimension. Let $Z$ and $V$ be the orthogonal bases of $\mathcal{Z}$ and $\mathcal{V}$ respectively. Then
\begin{eqnarray*}
\sin\theta &=& \sigma_{max}(Z^TV_\perp)=\sigma_{max}(V^TZ_\perp),\\
\cos\theta &=& \sigma_{min}(Z^TV)=\sigma_{min}(Z^T_\perp V_\perp).
\end{eqnarray*}
\end{lemma}\begin{proof}
The first identity and its proof can be found in \cite[Theorem 2.6.1]{Gene}. We only prove the second identity here. Since $(V,V_\perp)$ is orthogonal, it follows from $\|(V,V_\perp)^TZx\|_2=1$ for all unit 2-norm $x\in\mathbb{R}^r$ that $\|V^TZx\|^2_2+\|V^T_\perp Zx\|^2_2=1.$ Thus
\begin{eqnarray*}
\sigma_{min}(V^TZ)^2 &=& \min_{\|x\|_2=1}\|V^TZx\|^2_2=1-\max_{\|x\|_2=1}\|V^T_\perp Zx\|^2_2\\
                     &=& 1-\sigma_{max}(V^T_\perp Z)^2=\cos^2\theta.
\end{eqnarray*}
Similarly, since $(Z,Z_\perp)$ is orthogonal, it follows from $\|(Z,Z_\perp)^TV_\perp x\|_2=1$ for all unit 2-norm $x\in\mathbb{R}^{(n-r)}$ that $\|Z^TV_\perp x\|^2_2+\|Z^T_\perp V_\perp x\|^2_2=1.$ Thus
\begin{eqnarray*}
\sigma_{min}(Z^T_\perp V_\perp)^2 &=& \min_{\|x\|_2=1}\|Z^T_\perp V_\perp x\|^2_2=1-\max_{\|x\|_2=1}\|Z^TV_\perp x\|^2_2\\
                                &=& 1-\sigma_{max}(Z^TV_\perp)^2=\cos^2\theta.
\end{eqnarray*}
The second identity is therefore valid. 
\end{proof}

Let $P_D$ be defined by (\ref{de:P_D}). $P_DA$ is an symmetric matrix since $A$ is SPD. Hence, Courant-Fischer Minimax Theorem \cite{Gene} can be applied to estimate the eigenvalues of $P_DA$. As is known that if $A\in\mathbb{R}^{n\times n}$ is symmetric, then
$$\lambda_k(A)=\max_{dim(S)=k}\min_{x\in S,\|x\|_2=1}x^TAx,~~k=1,\dots,n.$$
\begin{theorem}\label{th:bo:P_DA}
Let $P_D$ be defined by (\ref{de:P_D}) and $\theta$ be the acute angle between $\mathcal{V}$ and $\mathcal{Z}$. Then $P_DA$ has $r$ zero eigenvalues. Moreover, if $\cos\theta\ne0$, then the nonzero eigenvalues of $P_DA$ satisfy
\begin{eqnarray*}
\lambda_{min}(\Lambda_\perp)-\varepsilon_D\le\lambda(P_DA)\le\lambda_{max}(\Lambda_\perp)+\eta_D,
\end{eqnarray*}
where $\eta_D=\lambda_{max}(\Lambda_\perp)(\sin\theta+\sin^2\theta)$ and
$\varepsilon_D = \eta_D+\|E^{-1}\|_2(\|E\|_2+\lambda_{max}(\Lambda_\perp))^2\tan^2\theta$.
\end{theorem}\begin{proof}
It follows from $P_DAZ=0$ that $P_DA$ has $r$ zero eigenvalues. From Theorem \ref{th:exact_coarse_space}, we have
\begin{eqnarray}\label{eq:xtP_DAx}
x^T\tilde P_DAx &=& x^TV_\perp\Lambda_\perp V^T_\perp x.
\end{eqnarray}

For all unit 2-norm $x\in\mathbb{R}^n$, it can be written as $x=x_1+x_2$, where $x_1\in\mathcal{Z}$ and $x_2\in\mathcal{Z}^\perp$. Moreover, there exist $t\in\mathbb{R}^r$ and $s\in\mathbb{R}^{n-r}$ such that $x_1=Zt$ and $x_2=Z_\perp s$, and
\begin{eqnarray*}
x^TAZE^{-1}Z^TAx &=& x^T_1AZE^{-1}Z^TAx_1+x^T_1AZE^{-1}Z^TAx_2\\
                 & & +x^T_2AZE^{-1}Z^TAx_1+x^T_2AZE^{-1}Z^TAx_2\\
                 &=& t^TZ^TAZt+t^TZ^TAx_2+x^T_2AZt\\
                 & & +x^T_2AZE^{-1}Z^TAx_2\\
                 &=& x^T_1Ax_1+x^T_1Ax_2+x^T_2Ax_1+x^T_2AZE^{-1}Z^TAx_2\\
                 &=& x^TAx-x^T_2Ax_2+x^T_2AZE^{-1}Z^TAx_2.
\end{eqnarray*}
Then we obtain
\begin{eqnarray}\label{eq:xP_DAx}
x^TP_DAx &=& x^TAx-x^TAZE^{-1}Z^TAx\\
         &=& x^T_2Ax_2-x^T_2AZE^{-1}Z^TAx_2.\nonumber
\end{eqnarray}
Subtract (\ref{eq:xtP_DAx}) from (\ref{eq:xP_DAx}) on both sides, we get
\begin{eqnarray}\label{eq:P_DA:eigenvalue}
x^TP_DAx = x^T\tilde P_DAx+x^T_2Ax_2-x^TV_\perp\Lambda_\perp V^T_\perp x-x^T_2AZE^{-1}Z^TAx_2.
\end{eqnarray}
The middle terms on the right-hand side of the above expression can be replaced by
\begin{eqnarray*}
x^T_2Ax_2-x^TV_\perp\Lambda_\perp V^T_\perp x &=& x^T_2Ax_2-(x^T_1+x^T_2)V_\perp\Lambda_\perp V^T_\perp(x_1+x_2)\\
                                         &=& x^T_2V\Lambda V^Tx_2-x^T_1V_\perp\Lambda_\perp V^T_\perp x_1\\
                                         & & -x^T_1V_\perp\Lambda_\perp V^T_\perp x_2-x^T_2V_\perp\Lambda_\perp V^T_\perp x_1\\
                                         &=& s^T(Z^T_\perp V)\Lambda(V^TZ_\perp)s-t^T(Z^TV_\perp)\Lambda_\perp(V^T_\perp Z)t\\
                                         & & -t^T(Z^TV_\perp)\Lambda_\perp(V^T_\perp Z_\perp)s-s^T(Z^T_\perp V_\perp)\Lambda_\perp(V^T_\perp Z)t.
\end{eqnarray*}
$\|x_1\|_2=\|t\|_2$ and $\|x_2\|_2=\|s\|_2$ since $Z$ and $Z_\perp$ are column orthogonal. Using Lemma \ref{le:angle}, we obtain
\begin{eqnarray}\label{bo:x2x}
|x^T_2Ax_2-x^TV_\perp\Lambda_\perp V^T_\perp x| &\le& (\|x_2\|^2_2\|\Lambda\|_2+\|x_1\|^2_2\|\Lambda_\perp\|_2)\sin^2\theta\\
&   & +2\|x_1\|_2\|x_2\|_2\|\Lambda_\perp\|_2\sin\theta\cos\theta\nonumber.
\end{eqnarray}

For the last item on the right-hand side of (\ref{eq:P_DA:eigenvalue}), we only need to estimate the bound of $Z^T_\perp AZ$ because of $x^T_2AZE^{-1}Z^TAx_2=s^T(Z^T_\perp AZ)E^{-1}(Z^T_\perp AZ)^Ts$. It follows from $(Z,Z_\perp)(Z,Z_\perp)^T=I$ that
$$AZ=Z(Z^TAZ)+Z_\perp(Z^T_\perp AZ).$$
Multiply by $V^T_\perp$ from the left on both sides of the above equation, we have
\begin{eqnarray*}
(V^T_\perp Z_\perp)(Z^T_\perp AZ) &=& V^T_\perp AZ-(V^T_\perp Z)(Z^TAZ)\\
                                  &=& \Lambda_\perp(V^T_\perp Z)-(V^T_\perp Z)(Z^TAZ).
\end{eqnarray*}
$V^T_\perp Z_\perp$ is invertible since $\cos\theta\ne0$. Thus
\begin{eqnarray}\label{eq:Z_perpAZ}
(Z^T_\perp AZ)=(V^T_\perp Z_\perp)^{-1}\Lambda_\perp(V^T_\perp Z)-(V^T_\perp Z_\perp)^{-1}(V^T_\perp Z)E.
\end{eqnarray}
Using Lemma \ref{le:angle}, we have $\|Z^T_\perp AZ\|_2\le(\|E\|_2+\|\Lambda_\perp\|_2)\tan\theta.$
Hence
\begin{eqnarray}\label{bo:upper:x2x2}
x^T_2AZE^{-1}Z^TAx_2\le\|x_2\|^2_2\|E^{-1}\|_2(\|E\|_2+\|\Lambda_\perp\|_2)^2\tan^2\theta.
\end{eqnarray}
$E^{-1}$ is SPD since $A$ is SPD. Consequently,
\begin{eqnarray}\label{bo:lower:x2x2}
x^T_2AZE^{-1}Z^TAx_2=(Z^TAx_2)^TE^{-1}(Z^TAx_2)\ge0.
\end{eqnarray}

$P_DA$ is symmetric since $A$ is SPD. Applying Courant-Fischer Minimax Theorem to (\ref{eq:P_DA:eigenvalue}) with (\ref{bo:x2x}) and (\ref{bo:lower:x2x2}), we have
\begin{eqnarray*}
\lambda(P_DA) &\le& \lambda(\tilde P_DA)+|x^T_2Ax_2-x^TV_\perp\Lambda_\perp V^T_\perp x|\\
&\le& \lambda_{max}(\Lambda_\perp)+\lambda_{max}(\Lambda_\perp)(\sin\theta+\sin^2\theta).
\end{eqnarray*}
Note that $P_DA$ has $r$ zero eigenvalues. Again, applying Courant-Fischer Minimax Theorem to (\ref{eq:P_DA:eigenvalue}) with (\ref{bo:x2x}) and (\ref{bo:upper:x2x2}), the lower bound of the nonzero eigenvalues of $P_DA$ is given by
\begin{eqnarray*}
\lambda(P_DA) &\ge& \lambda(\tilde P_DA)-|x^T_2Ax_2-x^TV_\perp\Lambda_\perp V^T_\perp x|-x^T_2AZE^{-1}Z^TAx_2\\
&\ge& \lambda_{min}(\Lambda_\perp)-\lambda_{max}(\Lambda_\perp)(\sin\theta+\sin^2\theta)\\
& & -\|E^{-1}\|_2(\|E\|_2+\lambda_{max}(\Lambda_\perp))^2\tan^2\theta.
\end{eqnarray*}
As a result the theorem is true.
\end{proof}

The above theorem shows that in exact arithmetic, as $\theta$ approaches zero, the maximal and minimal nonzero eigenvalues of $P_DA$ converge to $\lambda_{max}(\Lambda_\perp)$ and $\lambda_{min}(\Lambda_\perp)$, respectively. Hence with an appropriate coarse space the spectrum of $P_DA$ is similar to that of $\tilde P_DA$. When there exists rounding error, however, $P_DAZ$ may not be equal to a zero matrix. In this case, $P_DA$ possibly has some eigenvalues around zero that should be equal to zero in exact arithmetic. So there is a potential risk for $P_D$ to yield a poor spectrum of the preconditioned system.

The authors in \cite{TangNVE} investigated the properties of $P_D$, $P_{BNN}$ and $P_{ADEF1}$. They established the relations between these preconditioners in terms of the spectrum. Suppose that $M$ is an SPD matrix. Let the spectrum of $P_DM^{-1}A$ be given by $\{0,\dots,0,\gamma_{r+1},\dots,\gamma_n\}$ with $\gamma_{r+1}\le\gamma_{r+2}\le\cdots\le\gamma_n$. Let the spectrum of $P_{BNN}A$ and $P_{ADEF1}A$ be $\{1,\dots,1,\mu_{r+1},\dots,\mu_n\}$ with $\mu_{r+1}\le\mu_{r+2}\le\cdots\le\mu_n$. Then, $\gamma_i=\mu_i$ for all $i=r+1,\dots,n$. The proof of this result can be found in \cite[Theorem 3.3]{TangNVE}.
From the relation of $P_{D}$ and $P_{ADEF1}$, we immediately obtain the next corollary if we take $M=I$. 
\begin{corollary}\label{co:relation}
Let the spectrum of $P_DA$ be given by $\{0,\dots,0,\tilde\lambda_{r+1},\dots,\tilde\lambda_n\}$. Then the spectrum of $P_AA$ is $\{1,\dots,1,\tilde\lambda_{r+1},\dots,\tilde\lambda_n\}$.
\end{corollary}

Corollary \ref{co:relation} implies that if the spectrum of $P_DA$ is known, the one of $P_AA$ would be known. The spectral bounds of $P_DA$ are described in Theorem \ref{th:bo:P_DA}, so we can easily bound the spectrum of $P_AA$.
\begin{theorem}\label{th:bo:P_AA}
Let $P_A$ be defined by (\ref{de:P_A-DEF1}). Let $\theta$ the acute angle between subspaces $\mathcal{Z}$ and $\mathcal{V}$. If $\cos\theta\ne0$, then the eigenvalues of $P_AA$ satisfy $$\min\{1,\lambda_{min}(\Lambda_\perp)-\varepsilon_D\}\le\lambda(P_AA)\le\max\{1,\lambda_{max}(\Lambda_\perp)+\eta_D\},$$
where $\eta_D$ and $\varepsilon_D$ are defined in Theorem \ref{th:bo:P_DA}.
\end{theorem}
\begin{proof}
It follows directly from Theorem \ref{th:bo:P_DA} and Corollary \ref{co:relation}.
\end{proof}

Next, we consider the spectrum of $P_CA$. Note that $P_CA$ is not necessarily symmetric although both $A$ and $P_C$ are symmetric. So we can not apply Courant-Fischer Minimax theorem to estimate the eigenvalues of $P_CA$, which are positive since $P_CA$ is similar to a SPD matrix, and are a subset of $x^TP_CAx$ for all unit 2-norm $x\in\mathbb{R}^n$. Hence the eigenvalues of $P_CA$ can be bounded by estimating $x^TP_CAx$.
\begin{theorem}\label{th:bo:P_CA}
Let $P_C$ be defined by (\ref{de:P_C}), and $\theta$ the acute angle between $\mathcal{V}$ and $\mathcal{Z}$. If $\cos\theta\ne0$, then
\begin{eqnarray*}
\lambda_{max}(P_CA) &\le& \max\left\{1+\lambda_{max}(\Lambda), \lambda_{max}(\Lambda_\perp)\right\}+\varepsilon_C,\\
\lambda_{min}(P_CA) &\ge& \min\left\{1+\lambda_{min}(\Lambda), \lambda_{min}(\Lambda_\perp)\right\}-\varepsilon_C,
\end{eqnarray*}
where $\varepsilon_C=\frac{1}{2}(\lambda_{max}(\Lambda_\perp)\|E^{-1}\|_2+1)\tan\theta+\sin\theta+\sin^2\theta.$
\end{theorem}
\begin{proof}
$P_CA$ is similar to $A+A^{1/2}ZE^{-1}Z^TA^{1/2}$ since $A$ is SPD. Moreover $A+A^{1/2}ZE^{-1}Z^TA^{1/2}$ is SPD as well since $A$ and $E$ are SPD. Thus the eigenvalues of $P_CA$ are positive.

For all unit 2-norm $x\in\mathbb{R}^n$, we write $x=x_1+x_2$, where $x_1\in\mathcal{Z}$ and $x_2\in\mathcal{Z_\perp}$. There exists $t\in\mathbb{R}^r$ such that $x_1=Zt$, likewise, there is $s\in\mathbb{R}^{n-r}$ such that $x_2=Z_\perp s$. 

Then $x^TP_CAx$ can be expressed as follows
\begin{eqnarray}\label{eq:xP_CAx}
x^TP_CAx &=& x^TAx+(x_1+x_2)^TZE^{-1}Z^TA(x_1+x_2)\\
         &=& x^TAx+x^T_1ZE^{-1}Z^TAx_1+x^T_1ZE^{-1}Z^TAx_2\nonumber\\
         &=& x^TAx+x^T_1x_1+x^T_1ZE^{-1}Z^TAx_2.\nonumber
\end{eqnarray}
From the definition of $\tilde P_C$ in (\ref{th:exact_coarse_space}), we obtain
\begin{eqnarray}\label{eq:xtP_CAx}
x^T\tilde P_CAx &=& x^TAx+x^TV\tilde E^{-1}V^TAx\\
                &=& x^TAx+x^TVV^Tx.\nonumber
\end{eqnarray}
Subtract (\ref{eq:xtP_CAx}) from (\ref{eq:xP_CAx}), we have
\begin{eqnarray}\label{eq:P_CA:eigenvalue}
x^TP_CAx-x^T\tilde P_CAx &=& x^T_1ZE^{-1}Z^TAx_2+x^T_1x_1-x^TVV^Tx\\
                         &=& x^T_1ZE^{-1}Z^TAx_2+x^T_1V_\perp V^T_\perp x_1\nonumber\\
                         & & -x^T_2VV^Tx_2-2x^T_1VV^Tx_2.\nonumber
\end{eqnarray}
Since both $A$ and $E^{-1}$ are symmetric,
$$x^T_1ZE^{-1}Z^TAx_2=x^T_2AZE^{-1}Z^Tx_1=s^T_2Z^T_\perp AZE^{-1}t.$$
Using (\ref{eq:Z_perpAZ}), we have
\begin{eqnarray*}
(Z^T_\perp AZ)E^{-1}=(V^T_\perp Z_\perp)^{-1}\Lambda_\perp(V^T_\perp Z)E^{-1}-(V^T_\perp Z_\perp)^{-1}(V^T_\perp Z).
\end{eqnarray*}
Note that $\|x_1\|_2=\|t\|_2$ and $\|x_2\|_2=\|s\|_2$. Using Lemma \ref{le:angle}, we obtain
\begin{eqnarray}\label{eq:P_C:ZEZ_TA}
|x^T_1ZE^{-1}Z^TAx_2|\le\|x_1\|_2\|x_2\|_2(\|\Lambda_\perp\|_2\|E^{-1}\|_2+1)\tan\theta.
\end{eqnarray}
Since $\|x\|_2=1$, we have the following bounds with Lemma \ref{le:angle}
\begin{eqnarray}\label{ineq:vbounds}
0\le x^T_1V_\perp V^T_\perp x_1\le\|x_1\|^2_2\sin^2\theta,\nonumber\\
0\le x^T_2VV^Tx_2\le\|x_2\|^2_2\sin^2\theta,\\
|x^T_1VV^Tx_2|\le\|x_1\|_2\|x_2\|_2\sin\theta.\nonumber
\end{eqnarray}
With (\ref{eq:P_CA:eigenvalue}), (\ref{eq:P_C:ZEZ_TA}) and (\ref{ineq:vbounds}), we obtain
\begin{eqnarray*}
\lambda_{min}(\tilde P_CA)-\varepsilon_C \le x^TP_CAx \le \lambda_{max}(\tilde P_CA)+\varepsilon_C,
\end{eqnarray*}
where $\varepsilon_C=\frac{1}{2}(\lambda_{max}(\Lambda_\perp)\|E^{-1}\|_2+1)\tan\theta+\sin\theta+\sin^2\theta.$
Thus the theorem follows from $\min\{x^TP_CAx\}\le\lambda(P_CA)\le\max\{x^TP_CAx\}$ for all unit 2-norm $x\in\mathbb{R}^n$.
\end{proof}

From Theorem \ref{th:bo:P_CA}, we conclude that the maximal and minimal eigenvalues of $P_CA$ converge to $\max\{\lambda_{max}(\Lambda_\perp), 1+\lambda_{max}(\Lambda)\}$ and $\min\{\lambda_{min}(\Lambda_\perp), 1+\lambda_{min}(\Lambda)\}$ respectively as $\theta$ approaches zero. With an appropriate coarse space, the spectral distribution of $P_CA$ would be close to that of $\tilde P_CA$. Analogously, it can be proved that the maximal and minimal eigenvalues of $AP_C$ have the same bounds as that of $P_CA$.

\section{Inexact inverse of projection matrix}
In practice, we do not form $E^{-1}$ explicitly to compute $y=E^{-1}x$. Here $y$ and $x$ are vectors with the suitable size. Instead, we compute the LU factorization of $E$ once, then solve the two triangular linear systems to obtain $y$. But it is expensive to compute the LU factorization when the matrix $E$ is large. We therefore replace $E$ by a perturbed one that is cheaper to compute. In this seciton, we analyse how the perturbation in the projection matrix impacts the spectrum of the preconditioned matrix. Assume $\tilde H$ is an invertible matrix as an approximation to $\tilde E$ in (\ref{th:exact_coarse_space}).
\begin{theorem}\label{th:bo:tP_D:inexactE}
Let $\bar P_D=I-AV\tilde H^{-1}V^T$ and $\rho=\tilde E\tilde H^{-1}-I$. If the eigenvalues of $\bar P_DA$ are real, then
\begin{eqnarray*}
-\xi_D\le\lambda(\bar P_DA)\le\lambda_{max}(\Lambda_\perp)+\xi_D,
\end{eqnarray*}
where $\xi_D=\|\rho\|_2\|\Lambda\|_2.$
\end{theorem}
\begin{proof}
$\|\rho\|_2$ measures how much $\tilde H$ is approximate to $\tilde E$ when $\tilde H^{-1}$ is used as the right inverse of $\tilde E$. Obviously $\|\rho\|_2=0$ if and only if $\tilde E=\tilde H$. Assume $\|x\|_2=1$ and use the definition of $\tilde P_D$ in Theorem \ref{th:exact_coarse_space}, we obtain
\begin{eqnarray*}
x^T\bar P_DAx &=& x^TAx-x^TAV\tilde H^{-1}V^TAx\\
              &=& x^TAx-x^TVV^TAx-x^TV\rho V^TAx\nonumber\\
              &=& x^T\tilde P_DAx-x^TV\rho \Lambda V^Tx.\nonumber
\end{eqnarray*}
Since $|x^TV\rho\Lambda V^Tx|\le\|\rho\|_2\|\Lambda\|_2$,
$$-\|\rho\|_2\|\Lambda\|_2\le x^T\bar P_DAx \le\lambda_{max}(\tilde P_DA)+\|\rho\|_2\|\Lambda\|_2.$$
Since the eigenvalues of $\bar P_DA$ are real, the theorem follows from $\min\{x^T\bar P_DAx\}\le\lambda(\bar P_DA)\le\max\{x^T\bar P_DAx\}$.
\end{proof}

Theorem \ref{th:bo:tP_D:inexactE} states that $\bar P_DA$ might have small eigenvalues around zero if $\|\rho\|_2\ne 0$, which leads to the worse spectral distribution than that of $\tilde P_DA$.
\begin{theorem}\label{th:bo:tP_C:inexactE}
Let $\bar P_C=I+V\tilde H^{-1}V^T$ and $\rho=\tilde H^{-1}\tilde E-I$. If the eigenvalues of $\bar P_CA$ are real, then
\begin{eqnarray*}
\lambda_{max}(\bar P_CA) &\le& \max\left\{1+\lambda_{max}(\Lambda), \lambda_{max}(\Lambda_\perp)\right\}+\xi_C,\\
\lambda_{min}(\bar P_CA) &\ge& \min\left\{1+\lambda_{min}(\Lambda), \lambda_{min}(\Lambda_\perp)\right\}-\xi_C,
\end{eqnarray*}
where $\xi_C=\|\rho\|_2$.
\end{theorem}
\begin{proof}
In this theorem, we use $\tilde H^{-1}$ as the left inverse of $\tilde E$. Assume $\|x\|_2=1$ and use the definition of $\tilde P_C$ in Theorem \ref{th:exact_coarse_space}, then we have
\begin{eqnarray*}
x^T\bar P_CAx &=& x^TAx+x^TV\tilde H^{-1}V^TAx\\
              &=& x^TAx+x^TV\tilde E^{-1}V^TAx+x^TV\rho V^Tx\nonumber\\
              &=& x^T\tilde P_CAx+x^TV\rho V^Tx.\nonumber
\end{eqnarray*}
Since $|x^TV\rho V^Tx|\le\|\rho\|_2$,
$$\lambda_{min}(\tilde P_CA)-\|\rho\|_2\le x^T\bar P_CAx \le\lambda_{max}(\tilde P_CA)+\|\rho\|_2.$$
Since the eigenvalues of $\bar P_CA$ are real, the theorem follows from $\min\{x^T\bar P_CAx\}\le\lambda(\bar P_CA)\le\max\{x^T\bar P_CAx\}$.
\end{proof}

Theorem \ref{th:bo:tP_C:inexactE} implies that if the eigenvalues of $\bar P_CA$ are real, the spectral distribution of $\bar P_CA$ is a little influenced with small $\|\rho\|_2$, because the maximal and minimal eigenvalues of $\bar P_CA$ converge to $\max\{\lambda_{max}(\Lambda_\perp), 1+\lambda_{max}(\Lambda)\}$ and $\min\{\lambda_{min}(\Lambda_\perp), 1+\lambda_{min}(\Lambda)\}$ respectively as $\rho$ approaches zero. 

In next theorem, we need to estimate the difference of $\tilde H$ and $\tilde E$ when $\tilde H^{-1}$ is used as both the left and the right inverse of $\tilde E$.
\begin{theorem}\label{th:bo:tP_A:inexactE}
Let $\bar P_A=I-AV\tilde H^{-1}V^T+V\tilde H^{-1}V^T$. Let $\rho_1=\tilde E\tilde H^{-1}-I$ and $\rho_2=\tilde H^{-1}\tilde E-I$. If the eigenvalues of $\bar P_AA$ are real, then
\begin{eqnarray*}
\lambda_{max}(\bar P_AA) &\le& \max\left\{1, \lambda_{max}(\Lambda_\perp)\right\}+\xi_A,\\
\lambda_{min}(\bar P_AA) &\ge& \min\left\{1, \lambda_{min}(\Lambda_\perp)\right\}-\xi_A,
\end{eqnarray*}
where $\xi_A=\|\rho_1\|_2\|\Lambda\|_2+\|\rho_2\|_2$.
\end{theorem}
\begin{proof}
Assume $\|x\|_2=1$ and use the definition of $\tilde P_A$ in (\ref{th:exact_coarse_space}), we have
\begin{eqnarray*}
x^T\bar P_AAx &=& x^TAx-x^TAV\tilde H^{-1}V^TAx+x^TV\tilde H^{-1}V^TAx\\
              &=& x^TAx-x^TV\rho_1\Lambda V^Tx+x^TV\rho_2V^Tx\\
              & & -x^TAV\tilde E^{-1}V^TAx+x^TV\tilde E^{-1}V^TAx\\
              &=& x^T\tilde P_AAx-x^TV\rho_1\Lambda V^Tx+x^TV\rho_2V^Tx.
\end{eqnarray*}
Since $|x^TV\rho_1\Lambda V^Tx-x^TV\rho_2V^Tx|\le\|\rho_1\|_2\|\Lambda\|_2+\|\rho_2\|_2$,
$$\lambda_{min}(\tilde P_AA)-\|\rho_1\|_2\|\Lambda\|_2-\|\rho_2\|_2\le x^T\bar P_AAx\le\lambda_{max}(\tilde P_AA)+\|\rho_1\|_2\|\Lambda\|_2+\|\rho_2\|_2 .$$
We assumed that the eigenvalues of $\bar P_AA$ are real. Thus the theorem follows from $\min\{x^T\bar P_AAx\}\le\lambda(\bar P_AA)\le\max\{x^T\bar P_AAx\}$.
\end{proof}
%Similarly, Theorem \ref{th:bo:tP_A:inexactE} shows that the maximal and minimal eigenvalues of $\bar P_AA$ converge to $\max\{1, \lambda_{max}(\Lambda_\perp)\}$ and $\min\{1, \lambda_{min}(\Lambda_\perp)\}$ respectively as both $\|\rho_1\|_2$ and $\|\rho_2\|_2$ approach zero, if the eigenvalues of $\bar P_AA$ are real. This implies $\bar P_CA$ and $\tilde P_CA$ have the similar bounds of the spectrum if $\tilde H$ is a good approximation to $\tilde E$.
\section{Numerical experiments}
In this section, a numerical comparison of various preconditioners is reported. All tests are performed with Matlab (R2010b) on an Intel Core2 Duo E7500, 2.93GHz processor with 4Gb memory. Except for the deflation preconditioner, the system preconditioned by the coarse correction and adapted deflation preconditioners are not necessarily symmetric although $A$ is SPD. Therefore we apply GMRES \cite{Saad} to solve the preconditioned system iteratively. In addition, we apply Gram-Schmidt method with reorthogonalization to maintain the orthogonality of basis of the Krylov subspace \cite{Gene, Saad}.
\subsection{Diagonal matrix}
The first test case is a diagonal matrix with entries $10^{-7}$, $10^{-6}$, $\cdots$, $10^{-1}$, 1, 10, 10.1, 10.2, $\cdots$, 209, 209.1. The matrix has 7 small eigenvalues less than 1 to be removed. The eigenvectors associated with these eigenvalues are the unit vectors, i.e., $V=(e_1,e_2,\dots,e_7)$, where $e_i$ is the $i$th column of the identity matrix. The right-hand side is a vector of all ones. The initial guess vector is a zero vector. All tests are required to reduce the relative residual below $10^{-12}$. GMRES method without preconditioning converges at 273th iteration. The perturbations in the coarse space and the projection matrix are generated by the Matlab function rand.

Table \ref{ta:distance} shows the distance between the exact coarse space and the coarse space with various perturbations. It should be noted that it is expensive and unnecessary in practice to compute $\sin\theta$ by Lemma \ref{le:angle} for a general linear system. Assume that $(\tilde\lambda_i, \tilde v_i)$ $(i=1,\cdots,r)$ are Ritz pairs that are extracted from the perturbed coarse space by the Rayleigh-Ritz procedure \cite{MatrixAlgII}. In general, $\max_{i=1,\cdots,r}\{\|A\tilde v_i-\tilde\lambda_i\tilde v_i\|_2\}$ denoted by $res_{max}$ in Table \ref{ta:distance} decreases as the two subspaces approach each other. So we can use $res_{max}$ to measure the distance between the two subspaces since it is more convenient to compute.

\begin{table}[htbp!]
\caption{The distance between span\{$V$\} and span\{$V+rand/\varepsilon$\}.}
\centering\small
\begin{tabular}{|c|c|c|c|c|c|}\hline
& $\varepsilon=1e+01$ & $\varepsilon=1e+02$ & $\varepsilon=1e+03$ & $\varepsilon=1e+04$ & $\varepsilon=1e+05$ \\ \hline
$\sin\theta$ & 9.77e-01 & 5.06e-01 & 6.03e-02 & 6.05e-03 & 6.02e-04 \\ \hline
$res_{max}$ & 7.08e+01 & 5.54e+01 & 7.33     & 5.78e-01 & 4.43e-02 \\ \hline
\end{tabular}\label{ta:distance}
\end{table}

In Table \ref{ta:numberiteration-Z}, the second column shows the number of GMRES iterations with the three preconditioners in the case that there is no perturbation in the coarse space and the projection matrix. The columns 3-7 show the number of GMRES iterations when only the coarse space has some perturbation. As is shown, all preconditioners suffer from the perturbation if it is large (see the third column). As the perturbation decreases, $P_C$ and $P_A$ become better, whereas $P_D$ becomes better only when the perturbation is very small. On the other hand, if the perturbed coarse space is close enough to the exact one (see the last two columns), $P_D$ is slightly more efficient than $P_A$, and both of them are more efficient than $P_D$.

\begin{table}[htbp!]
\caption{The number of GMRES iterations with various preconditioners that are constructed with $Z=V+rand/\varepsilon$ and $E^{-1}$.}
\centering\small
\begin{tabular}{|c|c|c|c|c|c|c|}\hline
        & $V$, $\tilde E^{-1}$ & $\varepsilon=1e+01$ & $\varepsilon=1e+02$ & $\varepsilon=1e+03$ & $\varepsilon=1e+04$ & $\varepsilon=1e+05$ \\ \hline
 $P_D$  & 71  & $>$300        & $>$300          & $>$300          & 109             & 88  \\ \hline
 $P_C$  & 104 & 273           & 267             & 222             & 173             & 144 \\ \hline
 $P_A$  & 72  & 290           & 231             & 167             & 117             & 96  \\ \hline
\end{tabular}\label{ta:numberiteration-Z}
\end{table}

Figure \ref{fig:diagonal-spectrum-V1e3} and Figure \ref{fig:diagonal-spectrum-V1e5} show the eigenvalue distribution of the system preconditioned by the three preconditioners. As discussed in Section \ref{se:spectrum}, because of rounding error, $P_DA$ has some tiny eigenvalues around zero. In general, it is difficult to figure out the condition under which these tiny eigenvalues cause the stagnation in the convergence. For the test case of diagonal matrix, rounding error does not impact the convergence of $P_DA$ when the perturbation in the coarse space is significantly small. We also see that the spectrum of $P_AA$ is more clustered than that of $P_CA$, which is consistent with the estimated bounds of their spectrum described in Theorem \ref{th:bo:P_AA} and Theorem \ref{th:bo:P_CA}.

\begin{figure}[htbp!]
\begin{minipage}[t]{0.49\linewidth}
\centering
\includegraphics[width=\textwidth]{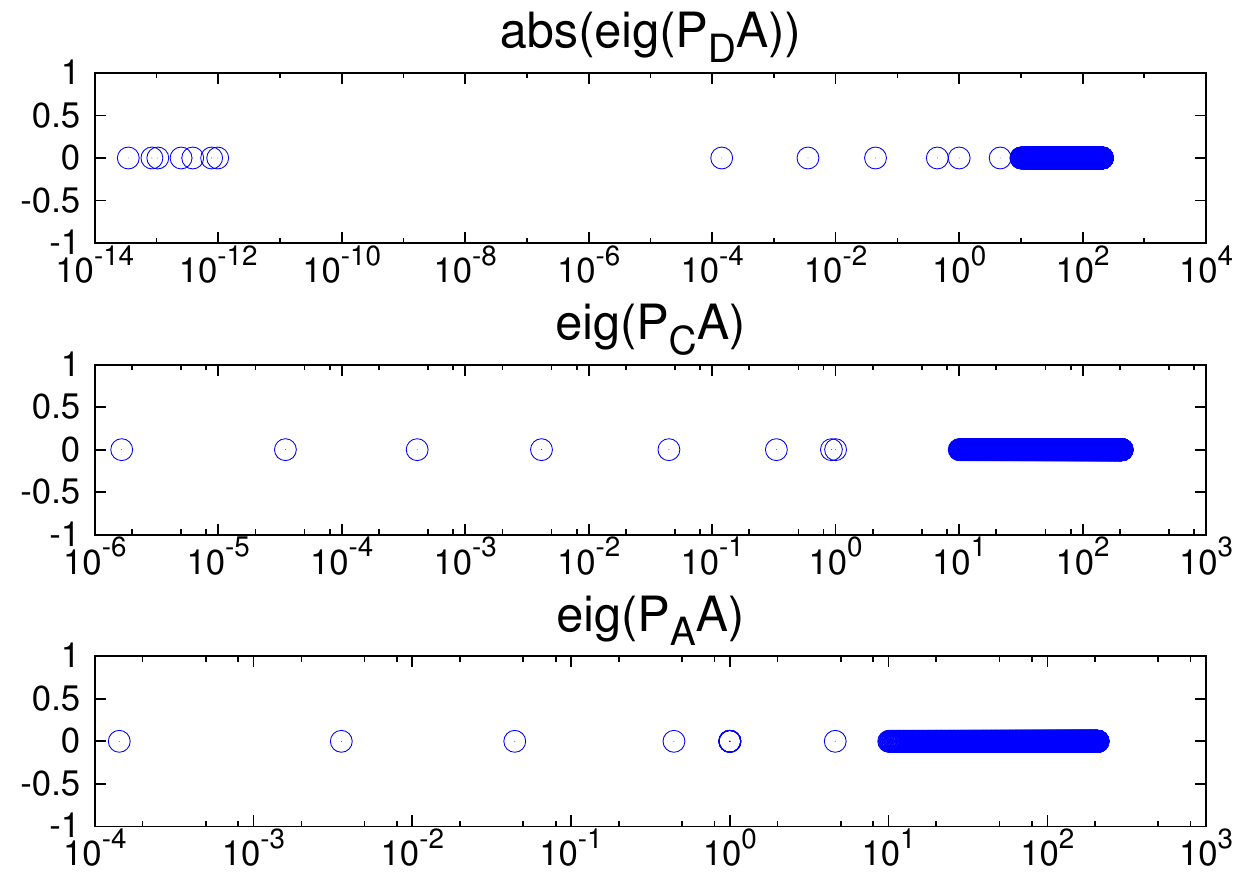}
\caption{The eigenvalue distribution of the preconditioned system. The preconditioners $P_D$, $P_C$ and $P_A$ are built with $Z=V+rand/1e+03$ and $E^{-1}$.}
\label{fig:diagonal-spectrum-V1e3}
\end{minipage}
\hfill
\begin{minipage}[t]{0.49\linewidth}
\centering
\includegraphics[width=\textwidth]{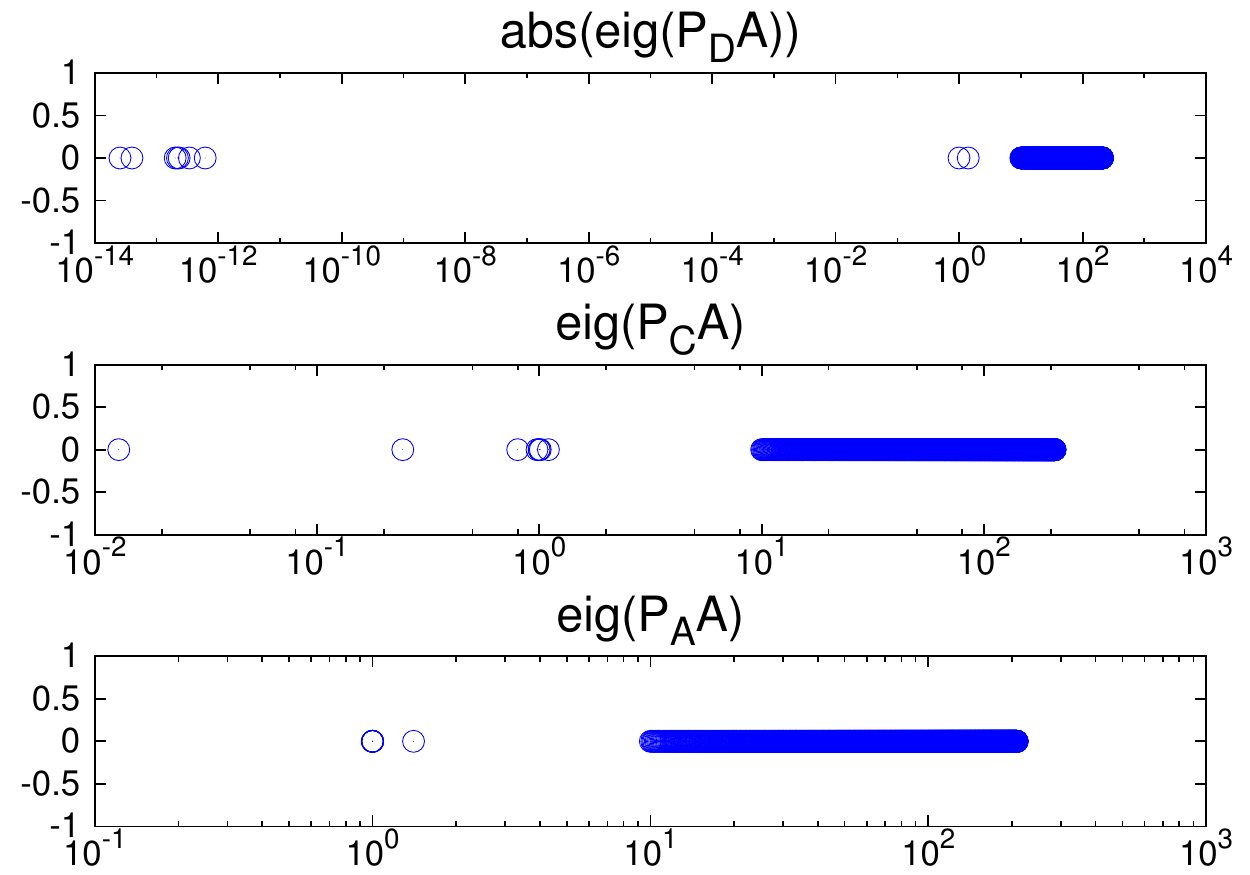}
\caption{The eigenvalue distribution of the preconditioned system. The preconditioners $P_D$, $P_C$ and $P_A$ are built with $Z=V+rand/1e+05$ and $E^{-1}$.}
\label{fig:diagonal-spectrum-V1e5}
\end{minipage}
\end{figure}

In Table \ref{ta:numberiteration-E}, the first three rows show the number of GMRES iterations when only the projection matrix has perturbation. The difference of $\tilde H^{-1}$ and $\tilde E^{-1}$ is reported in the last two rows. In comparison with the second column in Table \ref{ta:numberiteration-Z}, we conclude the perturbation in the projection matrix has a little impact on $P_C$ and $P_A$, but has a severe impact on $P_D$ even when $\tilde H^{-1}$ is very close to $\tilde E^{-1}$ (see the last column).

\begin{table}[htbp!]
\caption{The number of GMRES iterations with different preconditioners, where only the projection matrix is perturbed and the perturbation of $\tilde E$ is $\tilde H=\tilde E+rand/\varepsilon$.}
\centering\small
\begin{tabular}{|c|c|c|c|c|}\hline
        & $\varepsilon=1e+10$ & $\varepsilon=1e+12$ & $\varepsilon=1e+14$ & $\varepsilon=1e+16$ \\ \hline
 $P_D$  & $>$300        & $>$300          & $>$300          & $>$300    \\ \hline
 $P_C$  & 111           & 104             & 104             & 104    \\ \hline
 $P_A$  & 88            & 88              & 80              & 80     \\ \hline
 $\|\tilde H^{-1}\tilde E-I\|_2$ & 1.68e-03  & 1.08e-05 & 1.77e-07 & 1.23e-09 \\ \hline
 $\|\tilde E\tilde H^{-1}-I\|_2$ & 1.68e-03  & 1.08e-05 & 1.77e-07 & 1.23e-09 \\ \hline
\end{tabular}\label{ta:numberiteration-E}
\end{table}

Figure \ref{fig:diagonal-perturbation-E1e12} and Figure \ref{fig:diagonal-perturbation-E1e16} report the spectral distribution of $P_D$, $P_C$ and $P_A$. We see that $P_C$ and $P_A$ successfully shift the small eigenvalues of $A$ to around one. Due to the perturbation in the projection matrix, $P_D$ fails to deflate the small eigenvalues and thus yields some tiny eigenvalues around zero, which leads to a worse convergence (see the first row in Table \ref{ta:numberiteration-E}).

\begin{figure}[htbp!]
\begin{minipage}{0.49\linewidth}
\centering
\includegraphics[width=\textwidth]{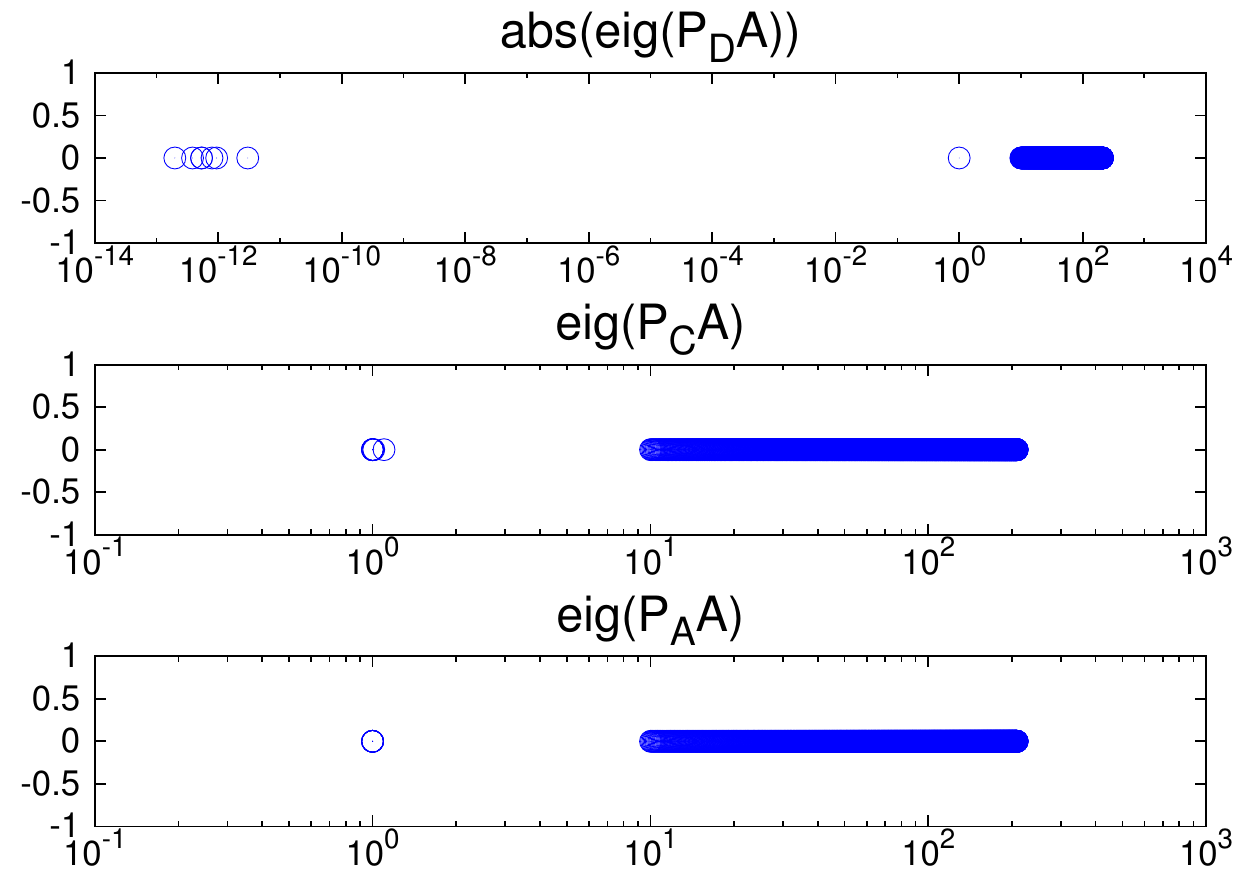}
\caption{The eigenvalue distribution of the preconditioned system, where the preconditioners $P_D$, $P_C$ and $P_A$ are built with $V$ and $\tilde H=\tilde E+rand/1e+12$.}
\label{fig:diagonal-perturbation-E1e12}
\end{minipage}
\hfill
\begin{minipage}{0.49\linewidth}
\centering
\includegraphics[width=\textwidth]{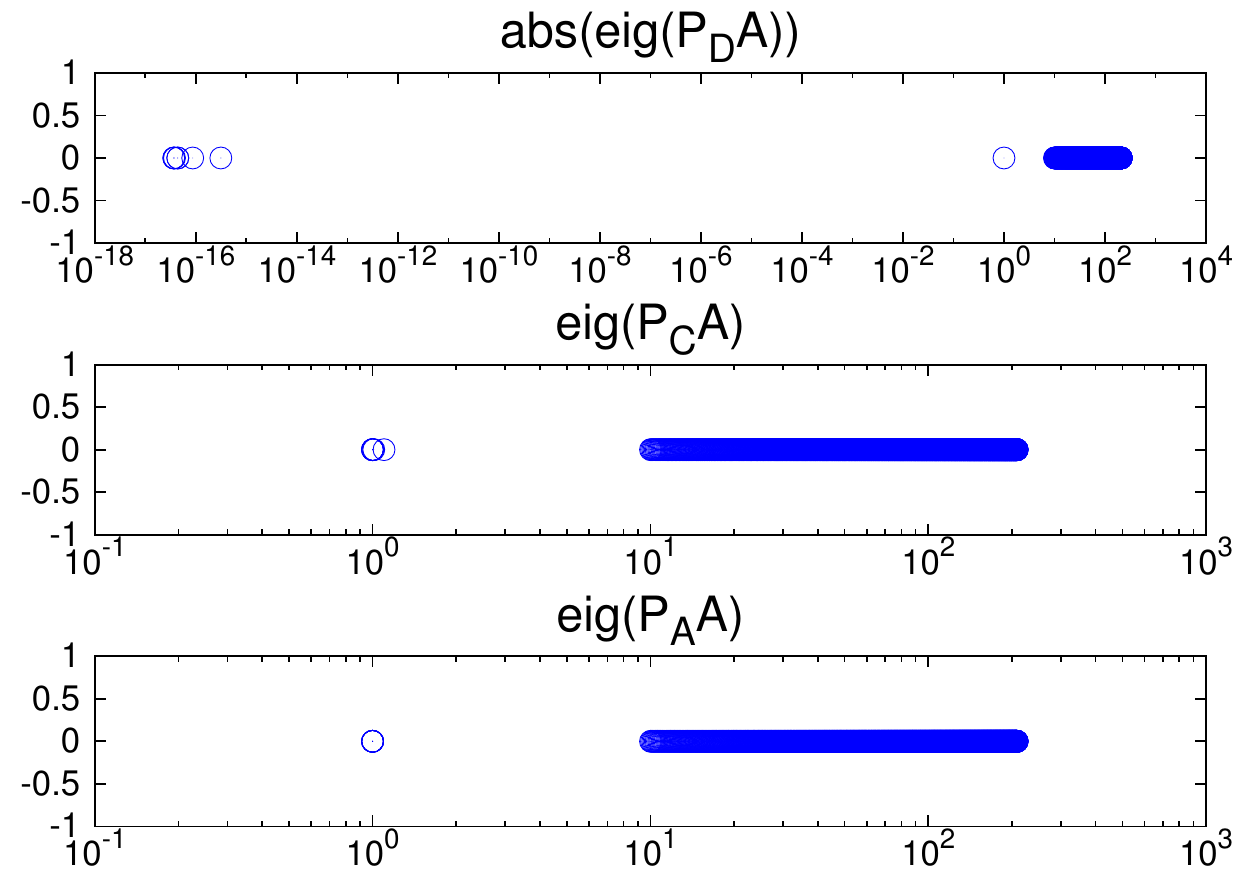}
\caption{The eigenvalue distribution of the preconditioned system, where the preconditioners $P_D$, $P_C$ and $P_A$ are built with $V$ and $\tilde H=\tilde E+rand/1e+16$.}
\label{fig:diagonal-perturbation-E1e16}
\end{minipage}
\end{figure}

\subsection{Boundary value problem}
We solve the following model problem
\begin{eqnarray*}
-\nabla\cdot(\kappa\nabla u) &=& f~~\mbox{in}~~\Omega=[0,1]^2,\\
u &=& 0~~\mbox{on}~~\partial\Omega,
\end{eqnarray*}
by the two-level multiplicative Schwarz method \cite{NatafMikolajZhao}. Here, $\kappa$ is the diffusion function of $x$ and $y$. The model problem is discretized by FreeFem++ \cite{Hecht} and the resulting coefficient matrix is of size 10201. Tests are performed on irregular overlapping decompositions with the overlap of 2 elements. These overlapping decompositions are built by adding the immediate neighboring vertices to non-overlapping subdomain obtained by Metis \cite{Kumar}.

In the two-level multiplicative Schwarz method, the first level preconditioner is the restricted additive Schwarz preconditioner (RAS) \cite{CaiSarkins} that is responsible to remove high frequency modes of the original system, and the deflation, coarse correction and adapted deflation preconditioners are applied as the second level preconditioners that remove lower frequency ones of the system preconditioned by one-level preconditioner \cite{NatafHuaVic, NatafMikolajZhao}.

We choose Ritz vectors to span the coarse space, which are extracted from the Krylov subspace during the solve of the system preconditioned by RAS. These vectors are the approximate eigenvectors corresponding to the lower part of the spectrum of the preconditioned system. To enrich the information on lower part of the spectrum, we construct the coarse space by splitting Ritz-vectors, see \cite{NatafHuaVic, NatafMikolajZhao, TangNVE} and references therein. More precisely, let
\begin{eqnarray*}
V = \left[\begin{array}{llll}
		v_{11} & v_{12} & \cdots & v_{1,r}   \\
		v_{21} & v_{22} & \cdots & v_{2,r}   \\
		\cdots & & & \\
        v_{nparts,1} & v_{nparts,2} & \cdots & v_{nparts,r}
	\end{array}\right]
\end{eqnarray*}
store Ritz vectors columnwise, where $nparts$ is the number of subdomains and $r$ the number of Ritz vectors; let $Z_i$ store the orthogonal vectors obtained by orthogonalizing $(v_{i1}, v_{i2}, \cdots, v_{i,r})$. Then $Z$ is formed as follows
\begin{eqnarray*}
Z = \left[\begin{array}{cccc}
		Z_1    & 0      & \cdots & 0      \\
		0      & Z_2    & \cdots & 0      \\
		\vdots & \vdots & \cdots & \vdots \\
        0      & 0      & \cdots & Z_{nparts}
	\end{array}\right].
\end{eqnarray*}

Obviously, span\{$V$\} is a subspace of span\{$Z$\}; span\{$Z$\} is $nparts$ times as large as span\{$V$\}. So span\{$Z$\} might have richer information corresponding to small eigenvalues. %Note that $Z$ to have a very sparse structure, which leads to a sparse structure of $E$ in (\ref{de:projection}) as well.

Some comparisons of the three preconditioners are performed on two different configurations with highly heterogeneous coefficient $\kappa$. See \cite{NatafHuaVic} for details. Two cases are described as following:
\begin{itemize}
\item skyscraper $\kappa$: for $x$ and $y$ such that for [9x]$\equiv$0(mod 2) and [9y]$\equiv$0(mod 2), $\kappa=10^4([9y]+1)$; and $\kappa=1$ elsewhere. See Figure \ref{fig:bvp_skyscraper}.
\item continuous $\kappa$: $\kappa(x,y)=10^6/3\sin(4\pi(x+y)+0.1)$. See Figure \ref{fig:bvp_continuous}.
\end{itemize}

For both cases, we use the zero vector as the initial guess vector. The iteration will stop when the relative residual is less than $10^{-10}$. Moreover, we construct the coarse space with all the eigenvectors associated with the eigenvalues less than 0.5 against the various domain decompositions.

\begin{figure}[htbp!]
\begin{minipage}[t]{0.48\linewidth}
\centering
\includegraphics[width=\textwidth]{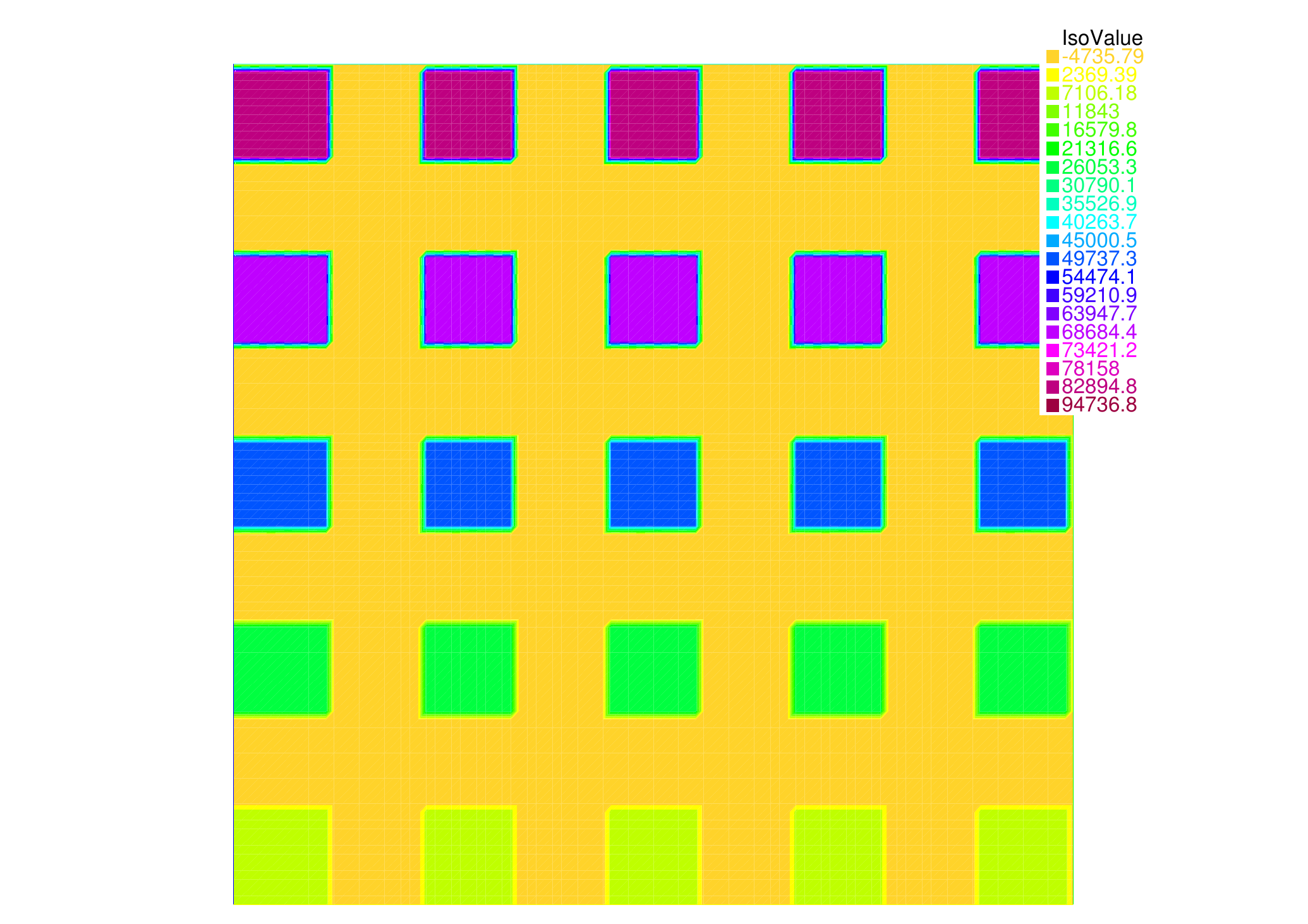}
\caption{Skyscraper case}
\label{fig:bvp_skyscraper}
\end{minipage}
\hfill
\begin{minipage}[t]{0.48\linewidth}
\centering
\includegraphics[width=\textwidth]{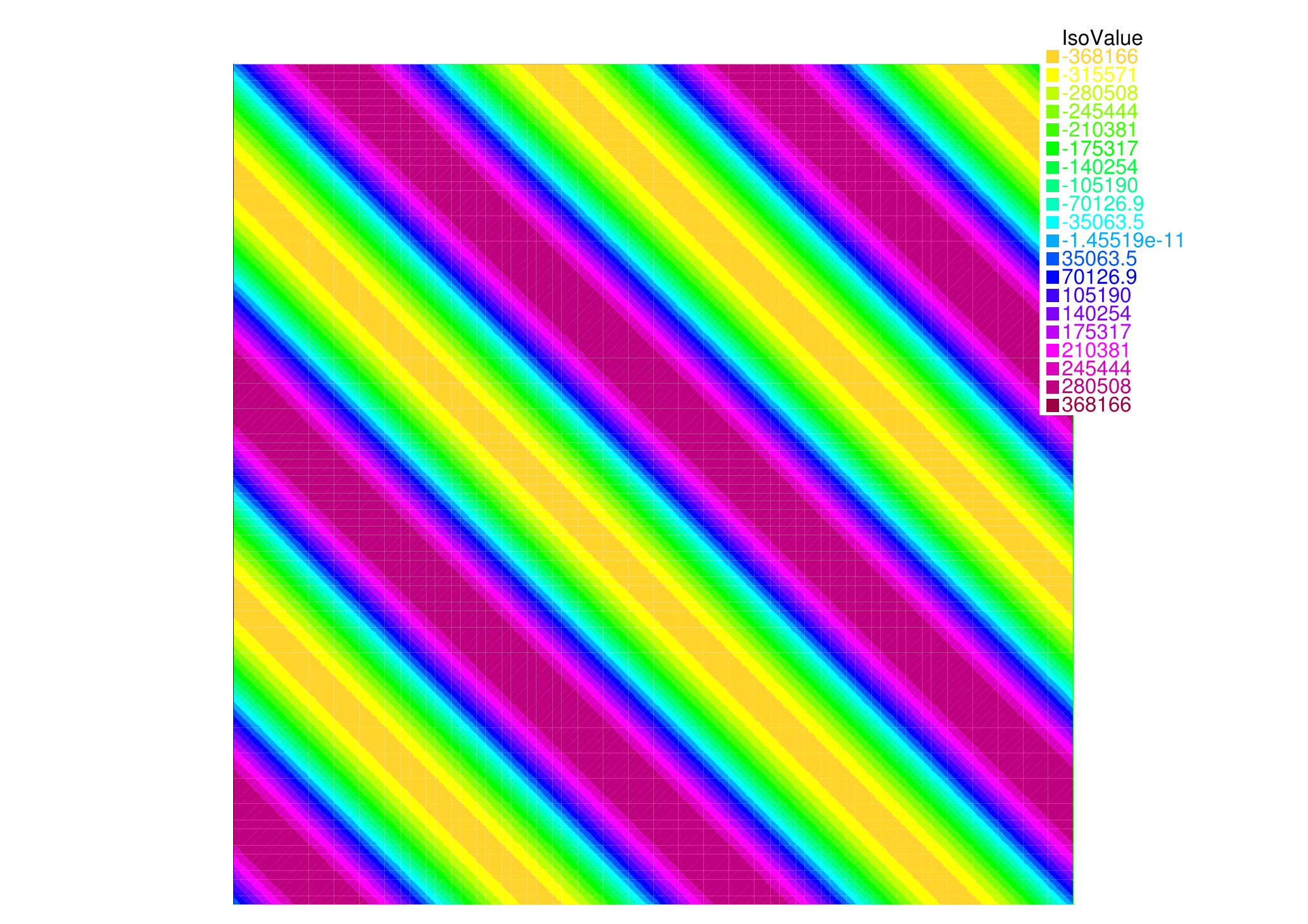}
\caption{Continuous case}
\label{fig:bvp_continuous}
\end{minipage}
\end{figure}

In Figures \ref{fig:bvp_s-Ritz14-part16}-\ref{fig:bvp_s-Ritz20-part128}, we report the convergence curves for the skycraper case against the various number of subdomains. Compared to one-level method, all three two-level methods improve convergence sufficiently. $P_D$ and $P_A$ have almost the same number of iterations although the initial residual of $P_D$ is much less than that of $P_A$. $P_D$ and $P_A$ are more efficient than $P_C$. Two-level method varies slightly on the number of iterations as the number of subdomains increases, while one-level method does.

\begin{figure}[htbp!]
\begin{minipage}[t]{0.49\linewidth}
\centering
\includegraphics[width=\textwidth]{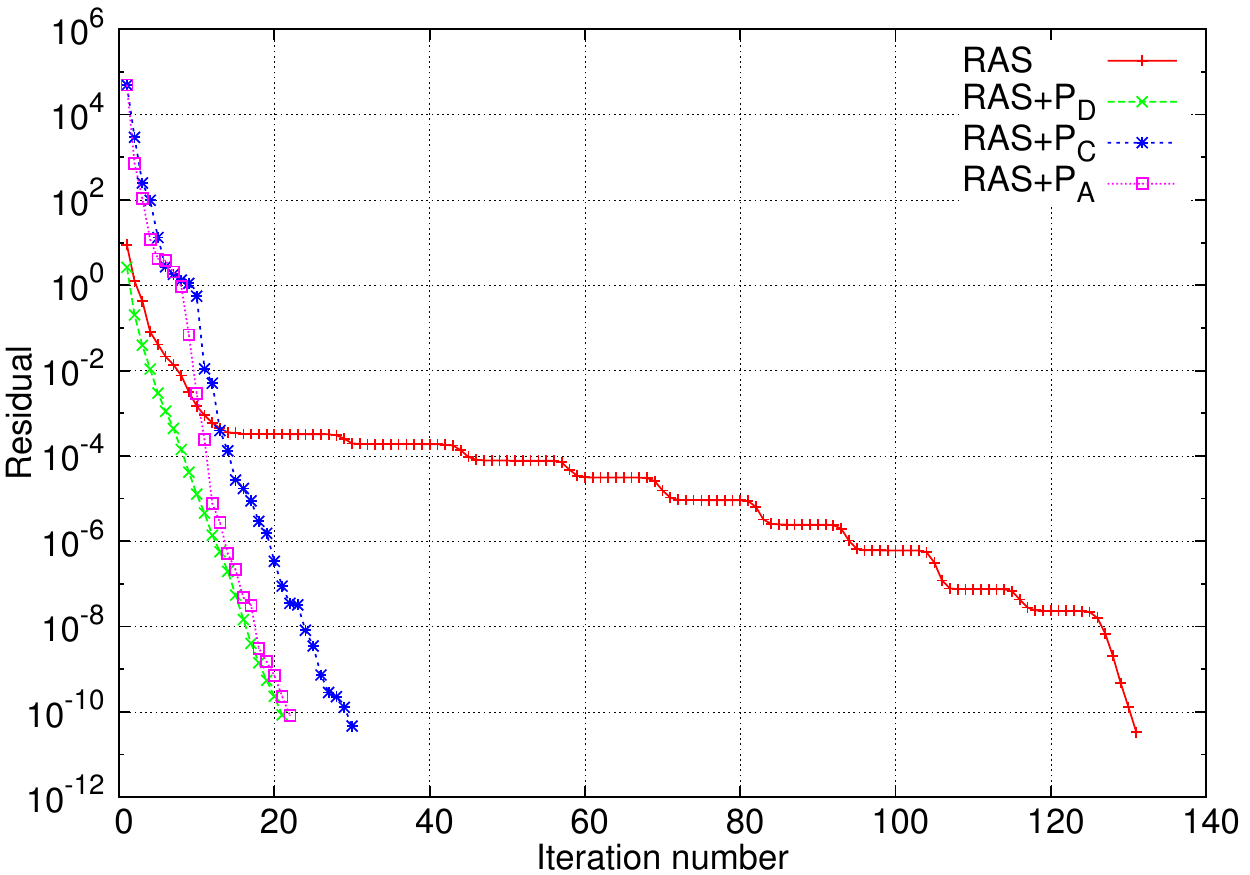}
\caption{Skyscraper case with 16 subdomains. 16 Ritz vectors spanning the coarse space.}
\label{fig:bvp_s-Ritz14-part16}
\end{minipage}
\hfill
\begin{minipage}[t]{0.49\linewidth}
\centering
\includegraphics[width=\textwidth]{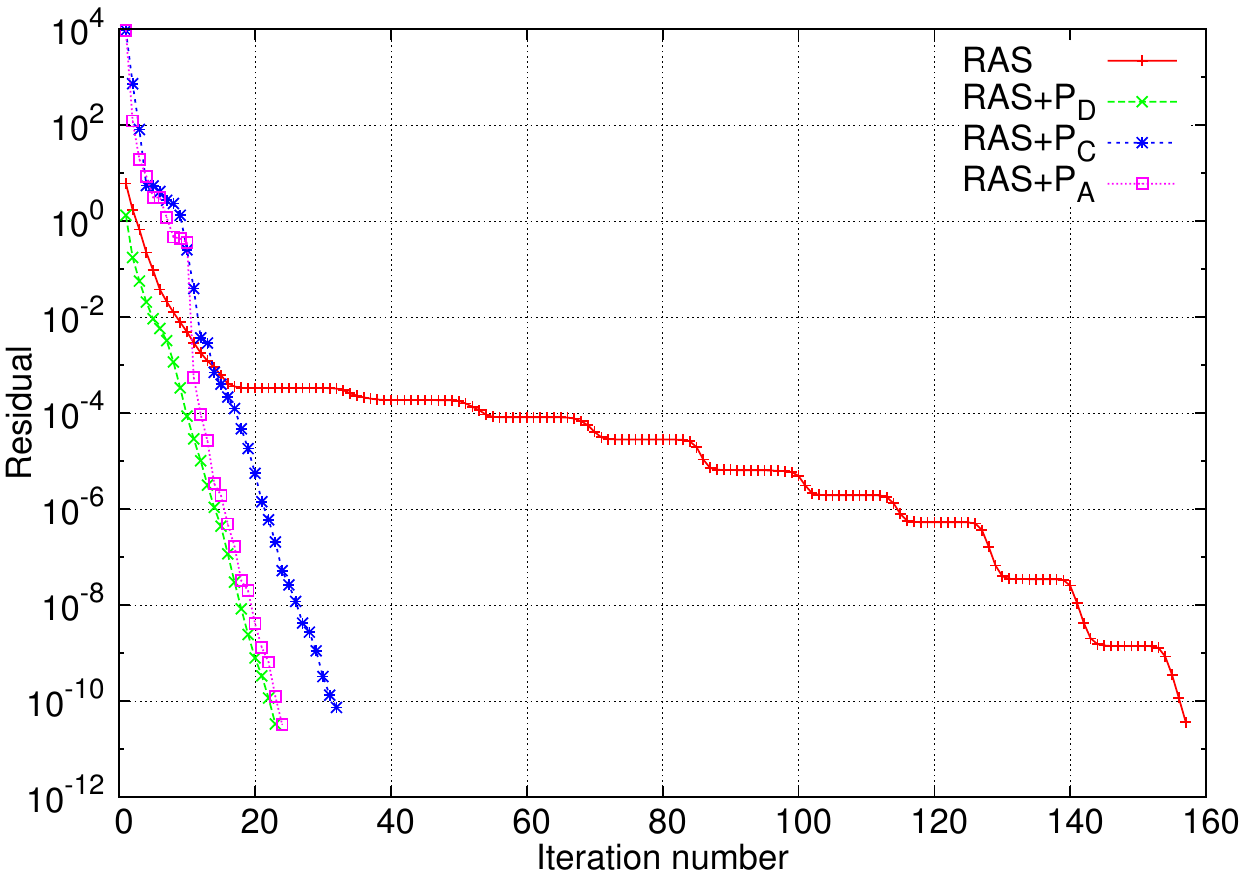}
\caption{Skyscraper case with 32 subdomains. 16 Ritz vectors spanning the coarse space.}
\label{fig:bvp_s-Ritz15-part32}
\end{minipage}
\end{figure}

\begin{figure}[htbp!]
\begin{minipage}[t]{0.49\linewidth}
\centering
\includegraphics[width=\textwidth]{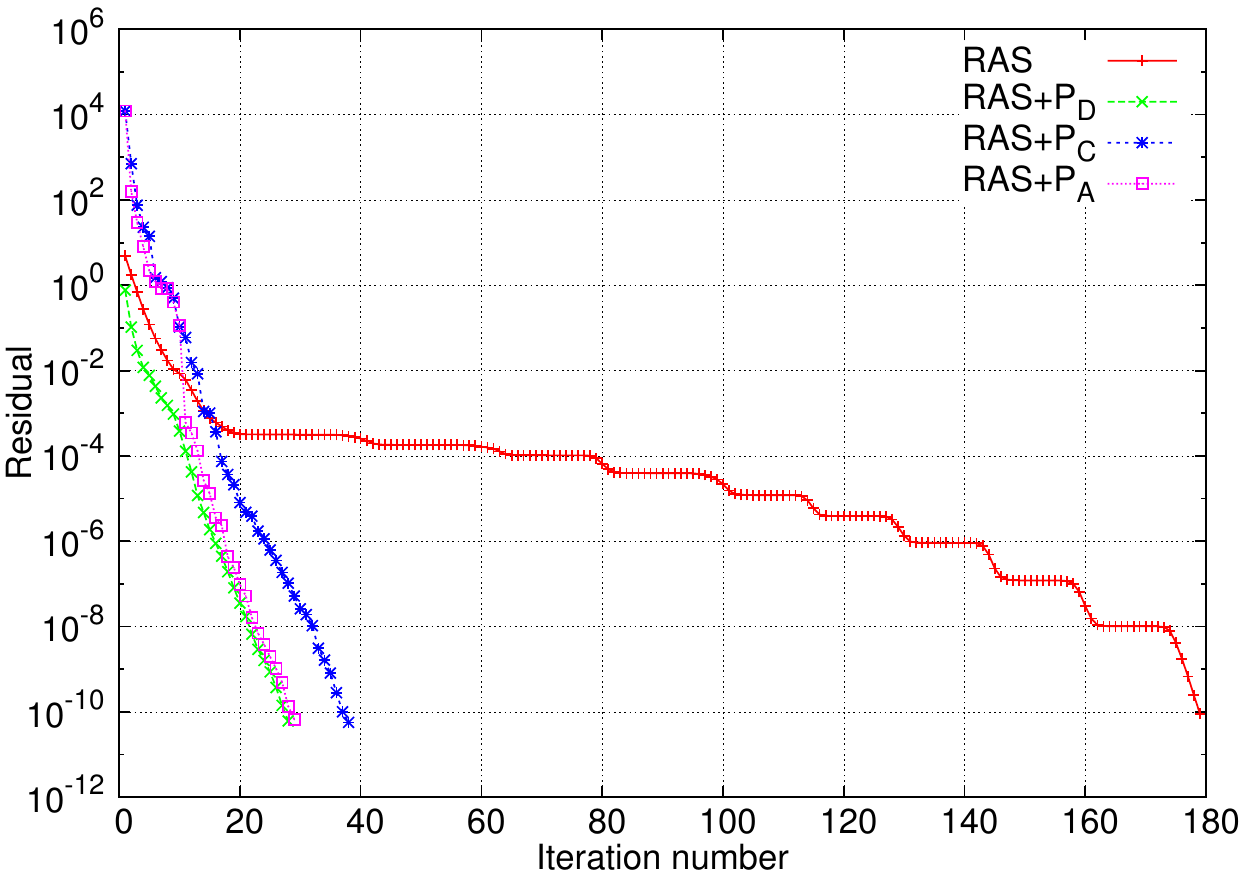}
\caption{Skyscraper case with 64 subdomains. 16 Ritz vectors spanning the coarse space.}
\label{fig:bvp_s-Ritz19-part64}
\end{minipage}
\hfill
\begin{minipage}[t]{0.49\linewidth}
\centering
\includegraphics[width=\textwidth]{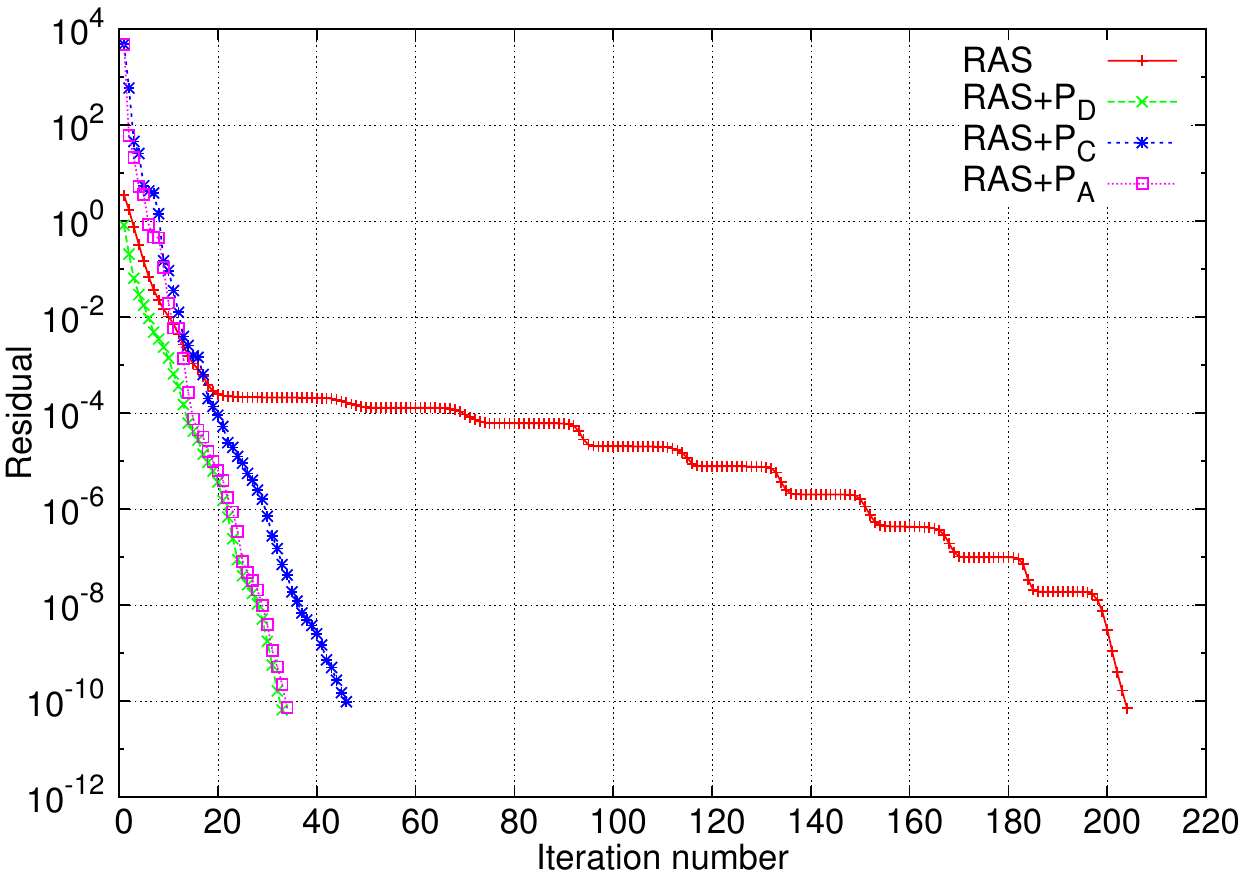}
\caption{Skyscraper case with 128 subdomains. 16 Ritz vectors spanning the coarse space.}
\label{fig:bvp_s-Ritz20-part128}
\end{minipage}
\end{figure}

Figures \ref{fig:bvp_c-Ritz16-part16}-\ref{fig:bvp_c-Ritz20-part128} plot the convergence curves for the continuous case against the various number of subdomains. Once again, we see that two-level method with RAS and the three preconditioners all outperform one-level method with only RAS. Like the skycraper case, $P_D$ and $P_A$ have almost the same number of iterations for all four decompositions and outperform $P_C$.

\begin{figure}[htbp!]
\begin{minipage}[t]{0.49\linewidth}
\centering
\includegraphics[width=\textwidth]{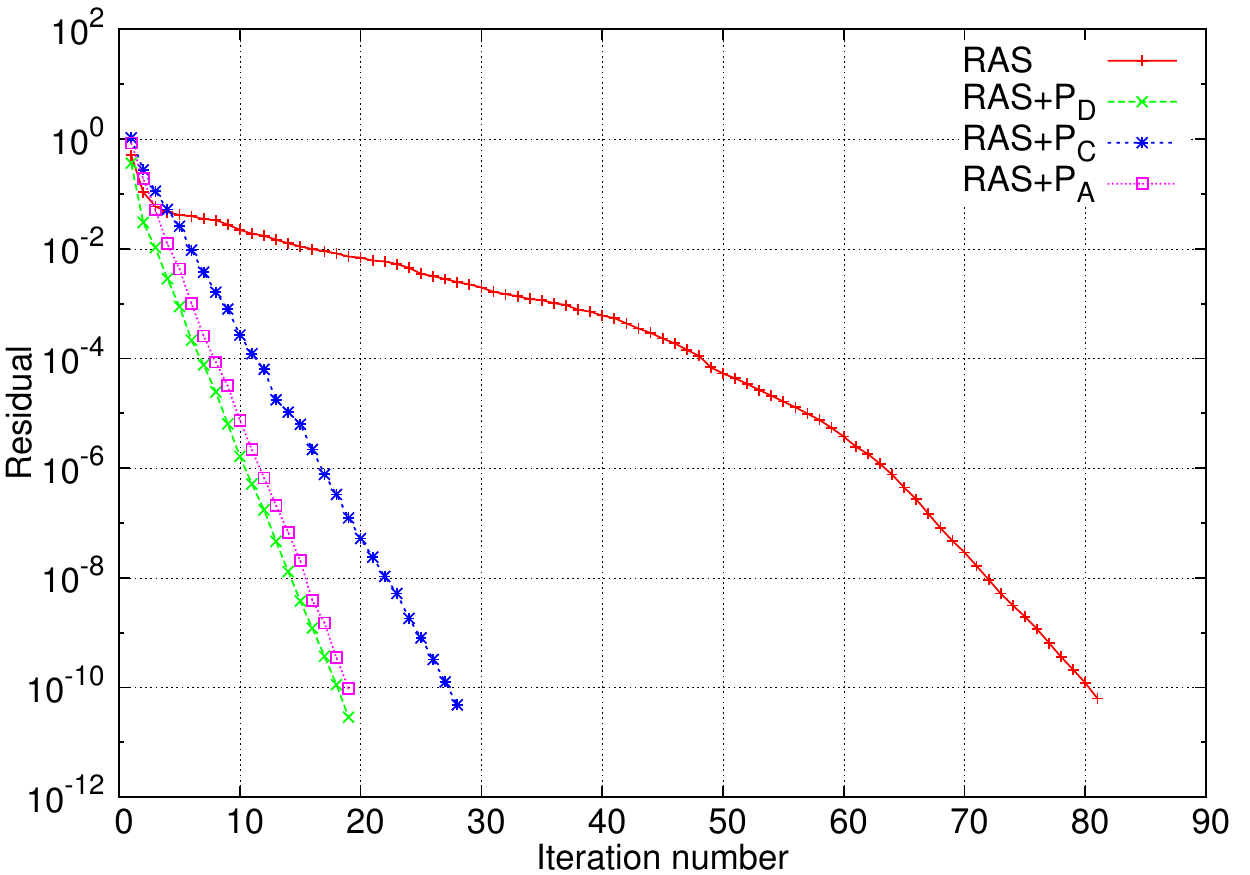}
\caption{Continuous case with 16 subdomains. 15 Ritz vectors spanning the coarse space.}
\label{fig:bvp_c-Ritz16-part16}
\end{minipage}
\hfill
\begin{minipage}[t]{0.49\linewidth}
\centering
\includegraphics[width=\textwidth]{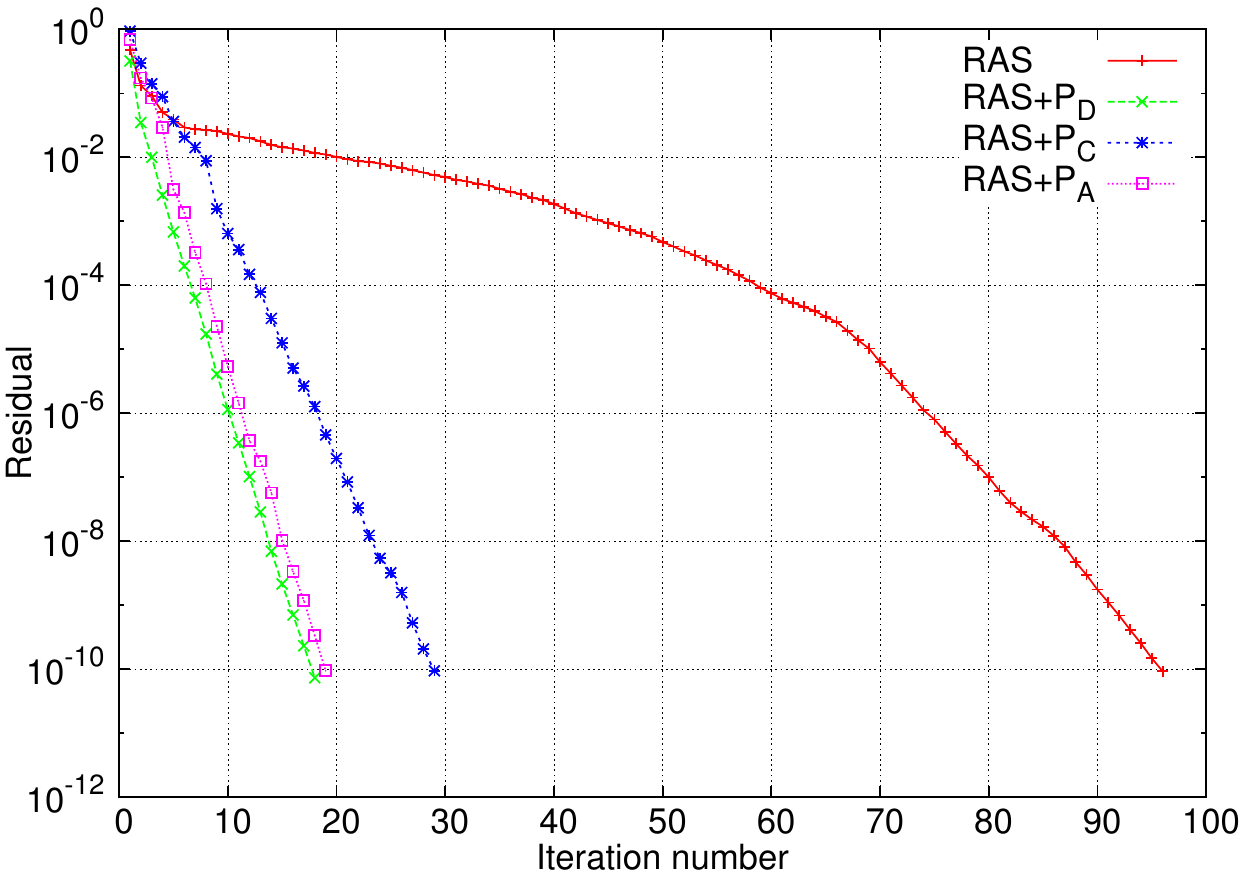}
\caption{Continuous case with 32 subdomains. 15 Ritz vectors spanning the coarse space.}
\label{fig:bvp_c-Ritz20-part32}
\end{minipage}
\end{figure}

\begin{figure}[htbp!]
\begin{minipage}[t]{0.49\linewidth}
\centering
\includegraphics[width=\textwidth]{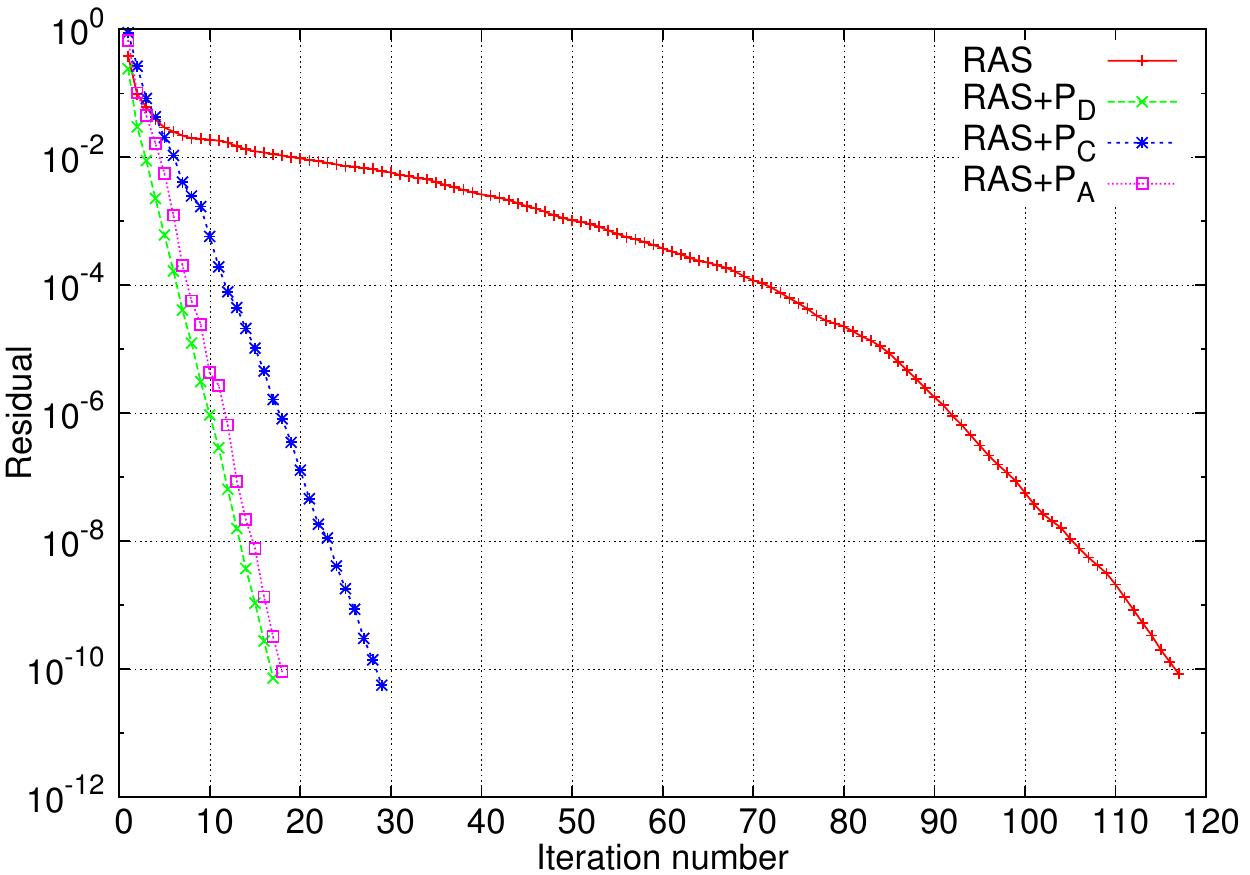}
\caption{Continuous case with 64 subdomains. 16 Ritz vectors spanning the coarse space.}
\label{fig:bvp_c-Ritz20-part64}
\end{minipage}
\hfill
\begin{minipage}[t]{0.49\linewidth}
\centering
\includegraphics[width=\textwidth]{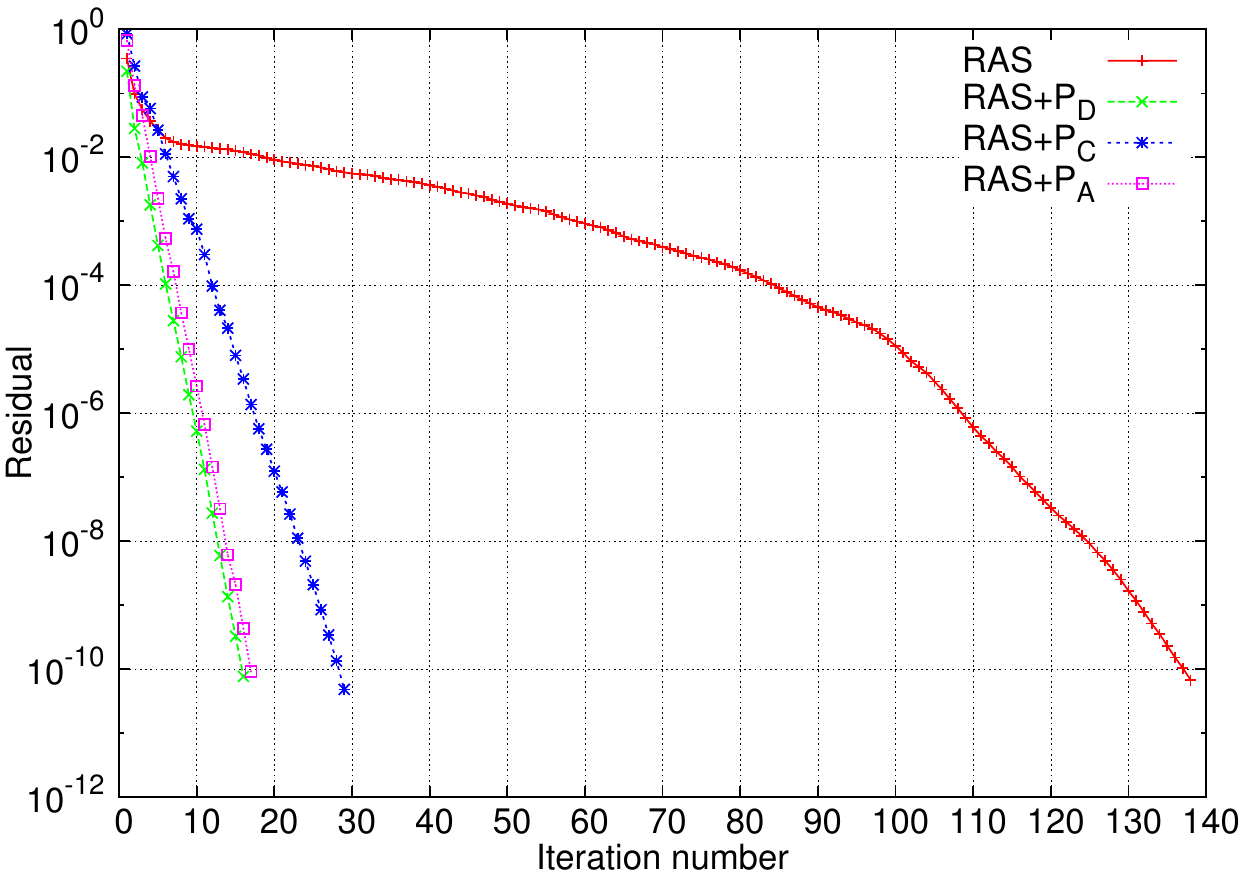}
\caption{Continuous case with 128 subdomains. 15 Ritz vectors spanning the coarse space.}
\label{fig:bvp_c-Ritz20-part128}
\end{minipage}
\end{figure}

In Table \ref{ta:residual}, we report the maximal residual of Ritz pairs for both cases that are extracted from the Krylov subspace when solving the system preconditioned only by RAS.

\begin{table}[htbp!]
\caption{Maximal residual of Ritz pairs for the skyscraper and continuous cases.}
\centering\small
\begin{tabular}{|c|c|c|c|c|}\hline
 Nparts     & 16 & 32 & 64 & 128 \\ \hline
 skyscraper & 2.564719e-09 & 9.162259e-08 & 1.577898e-08 & 4.747077e-09  \\ \hline
 continuous & 4.439121e-03 & 5.242112e-03 & 4.749042e-03 & 1.326165e-03  \\ \hline
\end{tabular}\label{ta:residual}
\end{table}

Note that the projection matrix $E$ is large in the case that the decomposition has 64 or 128 subdomains. As a consequence, computing LU factorization of $E^{-1}$ is costly and impairs the gains in the number of iterations. Hence, we attempt to compute the incomplete LU factorization of $E$, which is cheaper to compute than the LU factorization. Assume $L$ and $U$ are factors of an incomplete LU factorization with no fill-in (ILU(0)) of $E$. Note that $E$ has a sparse structure because of the sparse structure of $Z$. So $L$ and $U$ are sparse as well. This means that it is very cheap to solve $(LU)x=y$. In this way, $E$ is actually replaced by $LU$.

From Figure \ref{fig:bvp_s-InexactE-part64} and Figure \ref{fig:bvp_s-InexactE-part128}, it appears that the perturbation in $E$ has almost no impact on $P_A$ and $P_C$ for the skycraper case, but has severe impact on $P_D$ that leads to a stagnation in the convergence. Figure \ref{fig:bvp_c-InexactE-part64} and Figure \ref{fig:bvp_c-InexactE-part128} show that $P_A$ and $P_C$ are also stable for the continuous case. As is shown in Figure \ref{fig:bvp_c-InexactE-part128}, $P_D$ is still unstable when the perturbation in $E$ is not small enough (see the second column in Table \ref{ta:continuous}). However, Figure \ref{fig:bvp_c-InexactE-part64} shows that $P_D$ is stable, since the perturbation is small enough such that $LU$ is almost same as $E$ (see the first column in Table \ref{ta:continuous}). 

Table \ref{ta:skycraper} and Table \ref{ta:continuous} present the distance between $LU$ and $E$ for both cases, respectively.

\begin{figure}[htbp!]
\begin{minipage}[t]{0.48\linewidth}
\centering
\includegraphics[width=\textwidth]{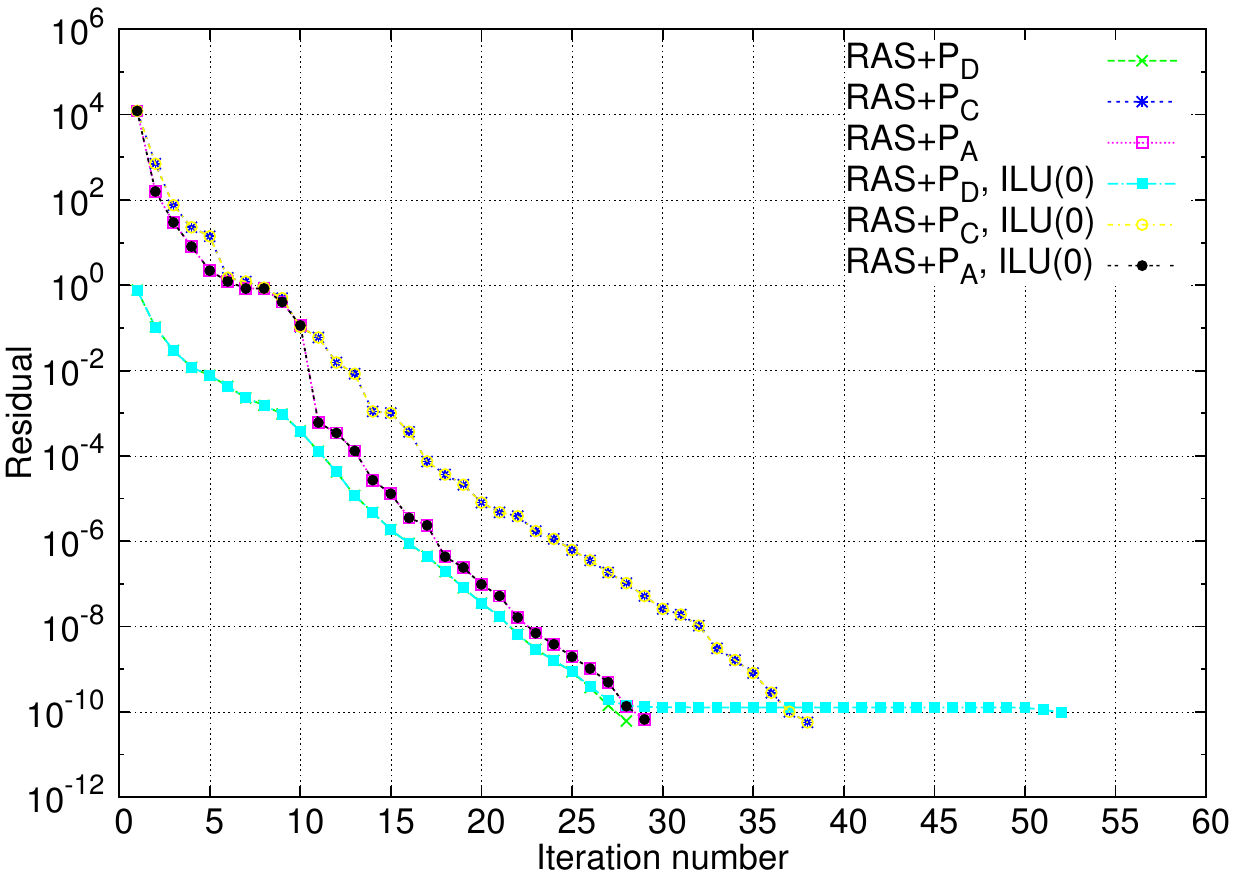}
\caption{Skyscraper case with 64 subdomains. $E$ is replaced by $LU$.}
\label{fig:bvp_s-InexactE-part64}
\end{minipage}
\hfill
\begin{minipage}[t]{0.48\linewidth}
\centering
\includegraphics[width=\textwidth]{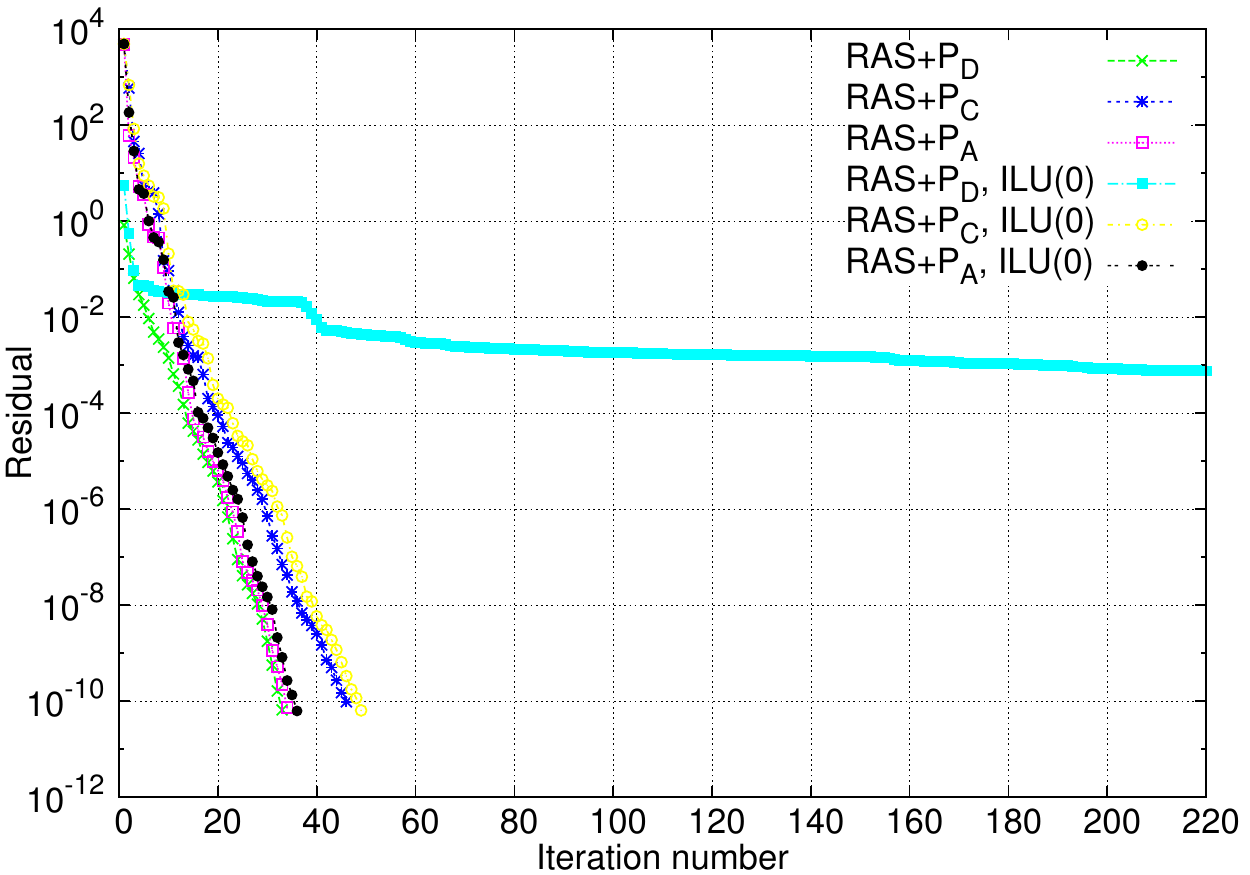}
\caption{Skyscraper case with 128 subdomains. $E$ is replaced by $LU$.}
\label{fig:bvp_s-InexactE-part128}
\end{minipage}
\end{figure}

\begin{table}[htbp!]
\caption{The distance between $LU$ and $E$ for the skyscraper case.}
\centering\small
\begin{tabular}{|c|c|c|}\hline
 Nparts             & 64         & 128 \\ \hline
 $\|E(LU)^{-1}-I\|_2$  & 3.8588e-08 & 8.9712e+02        \\ \hline
 $\|(LU)^{-1}E-I\|_2$  & 4.1433e-10 & 8.6292            \\ \hline
\end{tabular}\label{ta:skycraper}
\end{table}

\begin{figure}[htbp!]
\begin{minipage}[t]{0.48\linewidth}
\centering
\includegraphics[width=\textwidth]{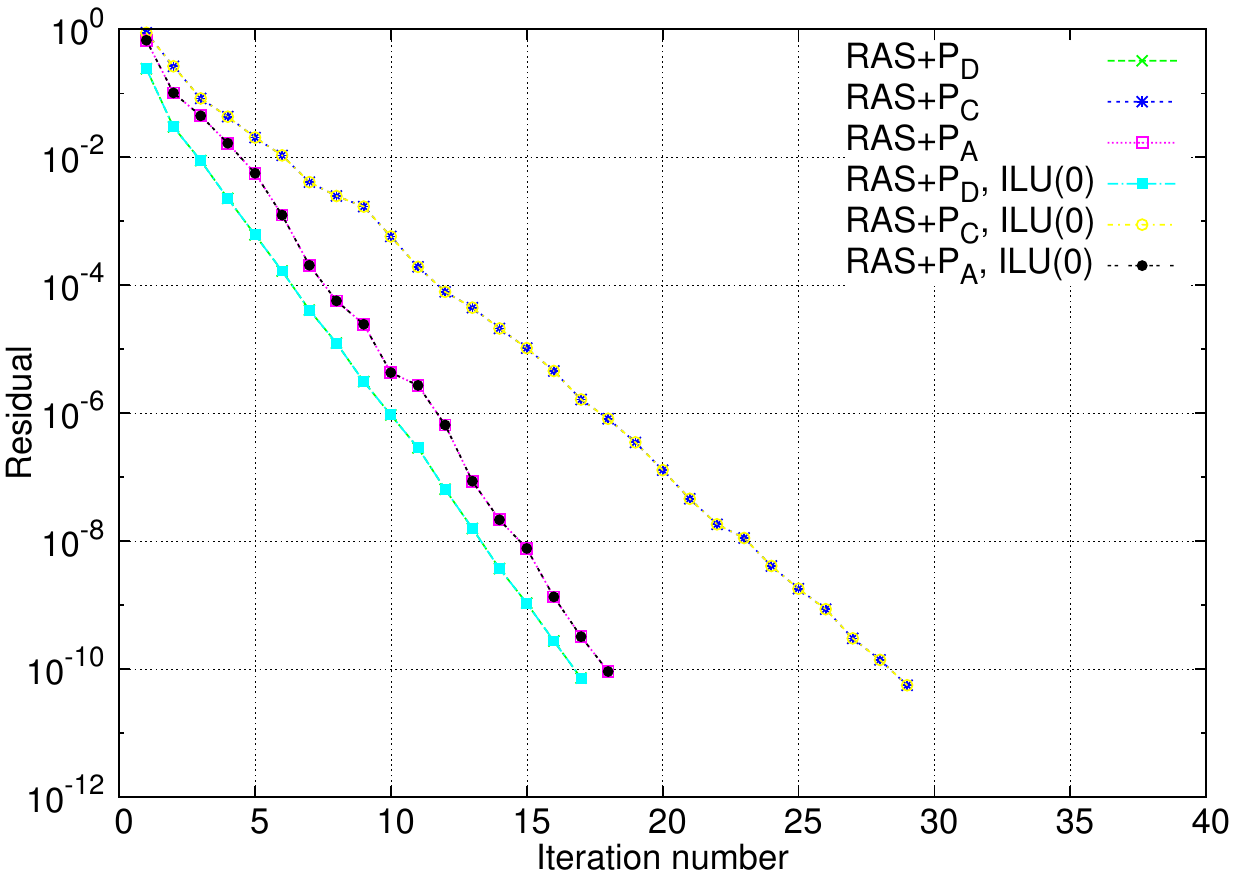}
\caption{Continuous case with 64 subdomains. $E$ is replaced by $LU$.}
\label{fig:bvp_c-InexactE-part64}
\end{minipage}
\hfill
\begin{minipage}[t]{0.48\linewidth}
\centering
\includegraphics[width=\textwidth]{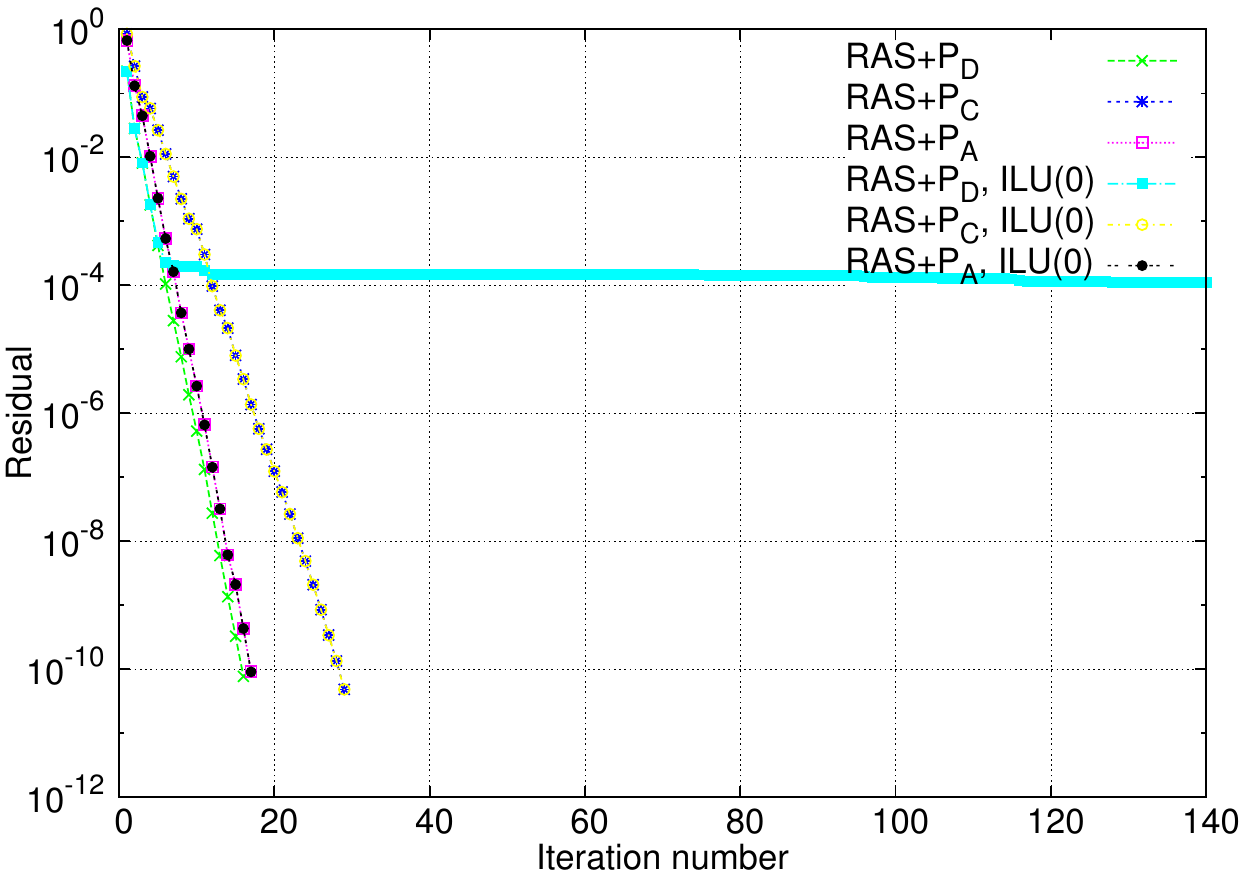}
\caption{Continuous case with 128 subdomains. $E$ is replaced by $LU$.}
\label{fig:bvp_c-InexactE-part128}
\end{minipage}
\end{figure}

\begin{table}[htbp!]
\caption{The distance between $LU$ and $E$ for the continuous case.}
\centering\small
\begin{tabular}{|c|c|c|}\hline
 Nparts             & 64         & 128 \\ \hline
 $\|E(LU)^{-1}-I\|_2$  & 4.1008e-14 & 5.4090e-01        \\ \hline
 $\|(LU)^{-1}E-I\|_2$  & 1.3890e-15 & 1.0215e-01        \\ \hline
\end{tabular}\label{ta:continuous}
\end{table}

\section{Conclusion}
We presented a perturbation analysis on the deflation, coarse correction and adapted deflation preconditioners when the inexact coarse space and inverse of projection matrix are applied for the construction of the preconditioners. Our analysis shows that in exact arithmetic the spectrum of the system preconditioned by the three preconditioners is impacted by the angle between the exact coarse space and the perturbed one. Moreover, we prove that the coarse correction and adapted deflation preconditioners are insensitive to the perturbation of the projection matrix, whereas the deflation preconditioner is sensitive. Numerical results of the different test cases confirm the perturbation analysis. 

{\bf Acknowledgements.}
The author would like to appreciate Professor Xiao-Chuan Cai for his valuable comments that are very helpful to improve the presentation of this paper.

\end{document}